\documentclass[11pt,final]{amsart}
\usepackage{amsopn}
\usepackage{amssymb, amscd}
\usepackage[notcite, notref]{showkeys}
\usepackage{enumerate}
\usepackage{tikz}
\usepackage{hyperref}
\usepackage{color}
\usepackage[all]{xy}

\newcommand{\nc}{\newcommand}

\nc{\fg}{\mathfrak{f} } \nc{\vg}{\mathfrak{v} } \nc{\wg}{\mathfrak{w} }
\nc{\zg}{\mathfrak{z} } \nc{\ngo}{\mathfrak{n} } \nc{\kg}{\mathfrak{k} }
\nc{\mg}{\mathfrak{m} } \nc{\bg}{\mathfrak{b} } \nc{\ggo}{\mathfrak{g} } \nc{\eg}{\mathfrak{e} }
\nc{\ggob}{\overline{\mathfrak{g}} } \nc{\sog}{\mathfrak{so} }
\nc{\sug}{\mathfrak{su} } \nc{\spg}{\mathfrak{sp} } \nc{\slg}{\mathfrak{sl} }
\nc{\glg}{\mathfrak{gl} } \nc{\cg}{\mathfrak{c} } \nc{\rg}{\mathfrak{r} }
\nc{\hg}{\mathfrak{h} } \nc{\tg}{\mathfrak{t} } \nc{\ug}{\mathfrak{u} }
\nc{\dg}{\mathfrak{d} } \nc{\ag}{\mathfrak{a} } \nc{\pg}{\mathfrak{p} }
\nc{\sg}{\mathfrak{s} } \nc{\affg}{\mathfrak{aff} } \nc{\qg}{\mathfrak{q} } \nc{\lgo}{\mathfrak{l} } \nc{\sping}{\mathfrak{spin} }

\nc{\pca}{\mathcal{P}} \nc{\nca}{\mathcal{N}} \nc{\lca}{\mathcal{L}}
\nc{\oca}{\mathcal{O}} \nc{\mca}{\mathcal{M}} \nc{\tca}{\mathcal{T}}
\nc{\aca}{\mathcal{A}} \nc{\cca}{\mathcal{C}} \nc{\gca}{\mathcal{G}}
\nc{\sca}{\mathcal{S}} \nc{\hca}{\mathcal{H}} \nc{\bca}{\mathcal{B}}
\nc{\dca}{\mathcal{D}} \nc{\eca}{\mathcal{E}} \nc{\wca}{\mathcal{W}}

\nc{\vp}{\varphi} \nc{\ddt}{\tfrac{d}{dt}} \nc{\dsdt}{\tfrac{d^2}{dt^2}} \nc{\dds}{\tfrac{d}{ds}}
\nc{\dpar}{\tfrac{\partial}{\partial t}} \nc{\im}{\mathrm{i}}

\nc{\SO}{\mathrm{SO}} \nc{\Spe}{\mathrm{Sp}} \nc{\Sl}{\mathrm{SL}}
\nc{\SU}{\mathrm{SU}} \nc{\Or}{\mathrm{O}} \nc{\U}{\mathrm{U}} \nc{\Gl}{\mathrm{GL}}
\nc{\Se}{\mathrm{S}} \nc{\Cl}{\mathrm{Cl}} \nc{\Spin}{\mathrm{Spin}}
\nc{\Pin}{\mathrm{Pin}} \nc{\G}{\mathrm{GL}_n(\RR)} \nc{\g}{\mathfrak{gl}_n(\RR)}

\nc{\RR}{{\mathbb R}} \nc{\HH}{{\mathbb H}} \nc{\CC}{{\mathbb C}} \nc{\ZZ}{{\mathbb Z}}
\nc{\FF}{{\mathbb F}} \nc{\NN}{{\mathbb N}} \nc{\QQ}{{\mathbb Q}} \nc{\PP}{{\mathbb P}} \nc{\OO}{{\mathbb O}}

\nc{\vs}{\vspace{.2cm}} \nc{\vsp}{\vspace{1cm}} \nc{\ip}{\langle\cdot,\cdot\rangle}
\nc{\ipp}{(\cdot,\cdot)} \nc{\la}{\langle} \nc{\ra}{\rangle} \nc{\unm}{\tfrac{1}{2}}
\nc{\unc}{\tfrac{1}{4}} \nc{\und}{\tfrac{1}{16}} \nc{\no}{\vs\noindent}
\nc{\lam}{\Lambda^2(\RR^n)^*\otimes\RR^n} \nc{\tangz}{{\rm T}^{\rm Zar}}
\nc{\nor}{{\sf n}}  \nc{\mum}{/\!\!/} \nc{\kir}{/\!\!/\!\!/}
\nc{\Ri}{\tfrac{4\Ric_{\mu}}{||\mu||^2}} \nc{\ds}{\displaystyle}
\nc{\ben}{\begin{enumerate}} \nc{\een}{\end{enumerate}} \nc{\f}{\tfrac}
\nc{\lb}{[\cdot,\cdot]} \nc{\isn}{\tfrac{1}{||v||^2}}
\nc{\gkp}{(\ggo=\kg\oplus\pg,\ip)} \nc{\ukh}{(\ug=\kg\oplus\hg,\ip)}
\nc{\tgkp}{(\tilde{\ggo}=\kg\oplus\pg,\ip)}
\nc{\wt}{\widetilde} 
\nc{\iop}{\mathtt{i}} \nc{\jop}{\mathtt{j}}

\nc{\Hess}{\operatorname{Hess}} \nc{\ad}{\operatorname{ad}}
\nc{\Ad}{\operatorname{Ad}} \nc{\rank}{\operatorname{rk}}
\nc{\Irr}{\operatorname{Irr}} \nc{\End}{\operatorname{End}}
\nc{\Aut}{\operatorname{Aut}} \nc{\Inn}{\operatorname{Inn}}
\nc{\Der}{\operatorname{Der}} \nc{\Ker}{\operatorname{Ker}}
\nc{\Iso}{\operatorname{Iso}} \nc{\Diff}{\operatorname{Diff}}
\nc{\Lie}{\operatorname{L}} \nc{\tr}{\operatorname{tr}} \nc{\dif}{\operatorname{d}}
\nc{\sen}{\operatorname{sen}} \nc{\modu}{\operatorname{mod}}
\nc{\CRic}{\operatorname{PP}} \nc{\Cric}{\operatorname{P}} \nc{\Ricci}{\operatorname{Ric}}
\nc{\sym}{\operatorname{sym}} \nc{\herm}{\operatorname{herm}} \nc{\symac}{\operatorname{sym^{ac}}}
\nc{\symc}{\operatorname{sym^{c}}} \nc{\scalar}{\operatorname{Sc}}
\nc{\grad}{\operatorname{grad}} \nc{\ricci}{\operatorname{Rc}} \nc{\kil}{\operatorname{B}} \nc{\cas}{\operatorname{C}} \nc{\lic}{\operatorname{L}}
\nc{\Nor}{\operatorname{Norm}}  \nc{\ricc}{\operatorname{Rc^{c}}}
\nc{\Ricc}{\operatorname{Ric^{c}}} \nc{\ricac}{\operatorname{Rc^{ac}}}
\nc{\Ricac}{\operatorname{Ric^{ac}}} \nc{\Riem}{\operatorname{Rm}} \nc{\Sec}{\operatorname{Sec}}
\nc{\riccig}{\operatorname{ric^{\gamma}}} \nc{\mm}{\operatorname{m}} \nc{\Mm}{\operatorname{M}}
\nc{\Le}{\operatorname{L}} \nc{\tang}{\operatorname{T}}
\nc{\level}{\operatorname{level}} \nc{\rad}{\operatorname{r}}
\nc{\abel}{\operatorname{ab}} \nc{\CH}{\operatorname{CH}} \nc{\Cone}{{\mathcal C}} \nc{\CCone}{\operatorname{CC}} \nc{\CP}{{\mathcal P}}
\nc{\mcc}{\operatorname{mcc}} \nc{\Adj}{\operatorname{Adj}}
\nc{\Order}{\operatorname{O}}  \nc{\inj}{\operatorname{inj}} \nc{\proy}{\operatorname{pr}}
\nc{\vol}{\operatorname{vol}} \nc{\Diag}{\operatorname{Dg}} \nc{\Diagg}{\operatorname{Diag}}
\nc{\Spec}{\operatorname{Spec}} \nc{\Ima}{\operatorname{Im}} \nc{\Rea}{\operatorname{Re}}
\nc{\spann}{\operatorname{span}} \nc{\Aff}{\operatorname{Aff}} \nc{\E}{\operatorname{E}} \nc{\id}{\operatorname{id}} \nc{\dete}{\operatorname{det}} \nc{\Crit}{\operatorname{Crit}} \nc{\val}{\operatorname{val}}

\theoremstyle{plain}
\newtheorem{theorem}{Theorem}[section]
\newtheorem{proposition}[theorem]{Proposition}
\newtheorem{corollary}[theorem]{Corollary}
\newtheorem{lemma}[theorem]{Lemma}

\theoremstyle{definition}
\newtheorem{definition}[theorem]{Definition}

\theoremstyle{remark}
\newtheorem{remark}[theorem]{Remark}
\newtheorem{example}[theorem]{Example}

\newcommand{\R}{\mathbb R}
\newcommand{\C}{\mathbb C}

\newcommand{\fe}{\mathfrak e}

\newcommand{\fh}{\mathfrak h}

\newcommand{\fk}{\mathfrak k}

\newcommand{\fp}{\mathfrak p}

\newcommand{\ft}{\mathfrak t}

\DeclareMathOperator{\Sp}{Sp}

\newcommand{\so}{\mathfrak{so}}

\newcommand{\su}{\mathfrak{su}}

\newcommand{\spin}{\mathfrak{spin}}

\newcommand{\Id}{\textup{Id}}

\newcommand{\mi}{\mathrm{i}}

\newcommand{\ee}{\varepsilon}

\newcommand{\bb}{b}
\newcommand{\nn}{m}
\newcommand{\II}{\mathbb{I}}

\newcommand{\midop}{\operatorname{mid}}

\title{The stability of standard homogeneous Einstein manifolds}

\author{Emilio A.~Lauret}  
	
\address{Instituto de Matem\'atica (INMABB), Departamento de Matem\'atica, Universidad Nacional del Sur (UNS)-CONICET, Bah\'ia Blanca, Argentina.}
\email{emilio.lauret@uns.edu.ar}

\author{Jorge Lauret}  
 
\address{FaMAF, Universidad Nacional de C\'ordoba and CIEM, CONICET (Argentina)}
\email{jorgelauret@unc.edu.ar} 

\thanks{This research was partially supported by grants from FONCyT, Univ.\ Nac.\ de C\'ordoba and Univ.\ Nac.\ del Sur, Argentina.}

\date{\today}

\begin{document}

\begin{abstract}
Back in 1985, Wang and Ziller obtained a complete classification of all homogeneous spaces of compact simple Lie groups on which the standard or Killing metric is Einstein.  The list consists, beyond isotropy irreducible spaces, of $12$ infinite families (two of them are actually conceptual constructions) and $22$ isolated examples.  We study in this paper the nature of each of these Einstein metrics as a critical point of the scalar curvature functional.  
\end{abstract}

\maketitle

\tableofcontents

\section{Introduction}\label{intro}

There is a canonical Riemannian metric on any homogeneous space $M^d=G/K$ of a compact semisimple Lie group $G$ naturally provided by the Killing form $\kil_\ggo$ of the Lie algebra $\ggo$ of $G$.   It is called the {\it standard} or {\it Killing} metric and will be denoted by $g_{\kil}$ in this paper.  Any simply connected compact homogeneous manifold $M$ therefore admits one standard metric for each of its presentations $M=G/K$ as a homogeneous space with $G$ compact semisimple.  

At this point, one might wonder how special is the standard metric from a geometric point of view.  When $M=G/K$ is an irreducible symmetric space, or more in general, an isotropy irreducible homogeneous space (see \cite{Wlf} or \cite[7.47]{Bss}), $g_{\kil}$ is the unique $G$-invariant metric on $M$ up to scaling and it is automatically Einstein.  Standard metrics have nonnegative sectional curvature, as actually any {\it normal} metric (i.e., defined by any bi-invariant inner product on $\ggo$) does.  All these metrics belong to the much wider class of {\it naturally reductive} metrics (with respect to $G$), that is when there is an $\Ad(K)$-invariant decomposition $\ggo=\kg\oplus\pg$ such that the one-parameter subgroups $\exp{tX}\cdot p$, $X\in\pg$ are all the geodesics through any point $p\in M$.  For $G$ simple, the three notions are known to coincide.  

The Einstein condition for $g_{\kil}$ turns out to be quite strong: the Casimir operator of the isotropy $K$-representation must act with the same multiple on each of the irreducible summands.  Indeed, the complete list obtained by Wang and Ziller in \cite{WngZll2} in the case when $G$ is simple, a real {\it tour de force} in representation theory, looks short considering the jungle of all possibilities.  The list consists, beyond isotropy irreducible spaces, of: 
\begin{enumerate}[{\small $\bullet$}]
\item $G$ classical (see Table \ref{tableIA}):
\begin{enumerate}[$\circ$] 
\item $10$ infinite families parametrized by the natural numbers, 

\item $2$ constructions parametrized by $l$-tuples ($l\geq 2$) of certain irreducible symmetric spaces listed in Table \ref{giki}, 

\item and $2$ isolated examples;
\end{enumerate}

\item $G$ exceptional (see Tables \ref{tableIB1}--\ref{tableIB3}): $20$ isolated examples.  
\end{enumerate}
The classification of standard Einstein metrics is still open for non-simple groups (see \cite[Section 4.14]{NknRdnSlv}).   

It is also natural to ask how special is a standard Einstein metric among the space $\mca^G$ of all $G$-invariant metrics on $M$.  Since $G$-invariant Einstein metrics are precisely the critical points of the scalar curvature functional 
$$
\scalar:\mca_1^G\longrightarrow \RR,
$$ 
where $\mca^G_1\subset\mca^G$ is the codimension one submanifold of all metrics of some fixed volume, the nature of $g_{\kil}$ as a critical point may give an insight.  

More in general, Einstein metrics on a compact differentiable manifold $M$ are the critical points of the total scalar curvature functional
\begin{equation*}\label{sct}
\widetilde{\scalar}(g):=\int_M \scalar(g)\; d\vol_g,
\end{equation*}
on the space $\mca_1$ of all Riemannian metrics on $M$ of some fixed volume (see \cite[4.21]{Bss}).  Moreover, if $\widetilde{\scalar}$ is further restricted to the submanifold
$$
\cca_1:=\{ g\in\mca_1:\scalar(g)\,\mbox{is a constant function on}\, M\}, 
$$ 
then the nullity and coindex of critical points are both finite (see \cite[4.60]{Bss}) and so the possibility of having a local maxima for $\widetilde{\scalar}|_{\cca_1}$ comes into play.  Remarkably, certain compact irreducible symmetric spaces are the only known local maxima of $\widetilde{\scalar}|_{\cca_1}$ with $\scalar>0$.

\begin{remark}\label{new}
After the first version of the present paper was uploaded to arXiv, the $G$-stable Einstein metric on $E_7/\SO(8)$ found in \cite{stab-dos} was proved to be stable in \cite{SchSmmWng}, giving the first example of a non-symmetric local maxima of $\widetilde{\scalar}|_{\cca_1}$ with $\scalar>0$ (see \cite{locmax} for further information).  
\end{remark}

The tangent space $T_g\cca_1$ coincides, modulo trivial variations, with the space $\tca\tca_g$ of all divergence-free (or transversal) and traceless symmetric $2$-tensors, so-called TT-{\it tensors} (see \cite[4.44-4.46]{Bss}), and if $\ricci(g)=\rho g$, then the Hessian of $\widetilde{\scalar}$ is given by
\begin{equation}\label{LL-intro}
\widetilde{\scalar}''_g(T,T) = \unm\la(2\rho\id-\Delta_L)T,T\ra_g, \qquad\forall T\in\tca\tca_g,
\end{equation}
where $\Delta_L$ is the {\it Lichnerowicz Laplacian} of $g$ (see \cite[4.64]{Bss}).  Analogously, in the $G$-invariant context, the tangent space $T_g\mca_1^G$ at a metric $g\in \mca_1^G$ is precisely the finite dimensional vector subspace $\tca\tca_g^G$  of $G$-invariant TT-tensors, modulo the trivial variations defined by $N_G(K)\subset\Diff(M)$, the normalizer of $K$ in $G$ (see \cite[Section 3.4]{stab-tres}).  All this naturally motivates the definition of the following concepts.  Let $\lambda_L$ (resp.\ $\lambda_L^G$) denote the smallest eigenvalue of $\Delta_L|_{\tca\tca_g}$ (resp.\ of $\lic_\pg:=\Delta_L|_{\tca\tca_g^G}$).

\begin{definition}\label{stab-def-intro} (see Figure \ref{stabtypes}).  
An Einstein metric $g\in\mca_1$ (resp.\ $g\in\mca_1^G$) with $\ricci(g)=\rho g$ is said to be,  
\begin{enumerate}[{\small $\bullet$}] 
\item {\it stable} (resp.\ $G$-{\it stable}): $\widetilde{\scalar}''_g|_{\tca\tca_g}<0$, i.e., $2\rho<\lambda_L$ (resp.\ $2\rho<\lambda_L^G$).  In particular, $g$ is a local maximum of $\widetilde{\scalar}|_{\cca_1}$ (resp.\ of $\scalar|_{\mca_1^G}$).  

\item {\it semistable} (resp.\ $G$-{\it semistable}): $\widetilde{\scalar}''_g|_{\tca\tca_g}\leq 0$, i.e., $2\rho\leq\lambda_L$ (resp.\ $2\rho\leq\lambda_L^G$); and otherwise {\it unstable} (resp.\ $G$-{\it unstable}).  

\item {\it neutrally stable} (resp.\ $G$-{\it neutrally stable}): $\widetilde{\scalar}''_g|_{\tca\tca_g}\leq 0$ and has nonzero kernel, i.e., $2\rho=\lambda_L$ (resp.\ $2\rho=\lambda_L^G$).     

\item \emph{non-degenerate} (resp.\ \emph{$G$-non-degenerate}): $\widetilde{\scalar}''_g|_{\tca\tca_g}$ (resp.\ $\scalar''_g|_{\tca\tca_g^G}$) is non-degenerate, i.e., $2\rho\notin\Spec(\Delta_L|_{\tca\tca_g})$ (resp.\ $2\rho\notin\Spec(\Delta_L|_{\tca\tca_g^G})$; otherwise, degenerate (resp.\ \emph{$G$-de\-gen\-er\-ate}).  In particular, if $G$ is non-degenerate (resp.\ $G$-non-degenerate) then $g$ is {\it rigid} (resp.\ $G$-{\it rigid}), in the sense that it is an isolated critical point up to the action of $\Diff(M)$ (resp.\ of $N_G(K)$). 

\item {\it dynamically stable}: for any metric $g_0$ near $g$, the normalized Ricci flow starting at $g_0$ exists for all $t\geq 0$ and converges modulo diffeomorphisms, as $t\to\infty$, to an Einstein metric near $g$ (see \cite{Krn2, Krn}).  

\item {\it $\nu$-stable}: $\nu''_g|_{C^\infty(M)g}< 0$ and $\nu''_g|_{\tca\tca_g}< 0$, where $\nu$ is the $\nu$-entropy functional introduced by Perelman, which is strictly increasing along Ricci flow solutions unless the solution consists of a shrinking gradient Ricci soliton like an Einstein metric with $\rho>0$ (see \cite{Prl, CaoHe,WngWng}).      
\end{enumerate}
\end{definition}

The study of the stability types of irreducible compact symmetric spaces was initiated by Koiso in \cite{Kso} and recently concluded in \cite{SemmelmannWeingart} and \cite{Schwahn}.  

In the $G$-invariant setting, the graph theorem \cite[Theorem 3.3]{BhmWngZll} and its generalization, the simplicial complex theorem \cite[Theorem 1.5]{Bhm}, suggest that $G$-instability is an expected behavior.  Indeed, all the Einstein metrics provided by these powerful existence results have positive {\it augmented coindex} (i.e., coindex plus nullity of $\scalar''_g|_{\tca\tca_g^G}$).  It is worth pointing out that, according to  
\cite[Theorem 5.1]{BhmWngZll}, a $G$-unstable Einstein metric does not realize the Yamabe invariant of $M$ (i.e., $\scalar(g)$ is not the supreme among all Yamabe metrics of $M$; recall that a metric is called {\it Yamabe} when it has the smallest scalar curvature in its unit volume conformal class).  There are several examples of unstable Einstein manifolds in the literature (see \cite{stab-dos} and references therein). 

The main result of this paper is the following.  

\begin{theorem}\label{main2}
Let $M=G/K$ be a homogeneous space which is not isotropy irreducible and assume that $G$ is simple and $g_{\kil}$ is Einstein.  Then the $G$-stability and critical point types of $g_{\kil}$ are given as in Tables \ref{tableIA}, \ref{tableIB1}, \ref{tableIB2} and \ref{tableIB3}, with the only exception of the space $\SO(4n^2)/\Spe(n)\times\Spe(n)$ in Table \ref{tableIA}, 3b.  
\end{theorem}

Some cases were previously solved in \cite{stab-tres, stab-dos, locmax, Nkn}.  The last column on the right of the tables provides a reference to the place in the paper (or to another paper) where the case was proved.  The smallest eigenvalue $\lambda_L^G$ of $\lic_\pg=\Delta_L|_{\tca\tca_g^G}$ is denoted by $\lambda_\pg$ and the greatest one by $\lambda_\pg^{\max}$ in the tables (see \S\ref{main-sec} for further information on the notation used in the tables).  

A main ingredient in the proof of Theorem \ref{main2} is the following stability criterion for $g_{\kil}$ in terms of only the extremal Casimir eigenvalues $\lambda_\tau$ and $\lambda_\tau^{\max}$ of the representation $\sym(\ggo)$ of $\ggo$ (see Table \ref{table1}), without even involving the Einstein constant $\rho$.  

\begin{theorem}\label{sc-intro} \textup{(See Theorem \ref{sc2}, (i))}.  
The standard metric $g_{\kil}$ on $G/K$ is $G$-stable if 
\begin{equation*}
\frac{\dim{K}}{\dim{G}}<\frac{\lambda_\tau-1}{2\lambda_\tau^{\max}}.
\end{equation*}
\end{theorem}

The proof of this purely Lie theoretical criterion is based on the formula for $\lic_\pg$ given in \cite[Section 5]{stab-tres} and most cases with $\ggo$ exceptional follow from it.   We also strongly use throughout the paper the formula for the matrix of $\lic_\pg$ in terms of structural constants given in \cite{stab-tres, stab-dos}.  

We now discuss some geometric consequences of Theorem~\ref{main2}.
When $\dim{\mca_1^G}>1$, local maxima of $\scalar|_{\mca_1^G}$ are not common.  The local maxima of $\scalar|_{\mca_1^G}$ (not necessarily standard) which had been known before all have either that $K$ is a maximal subgroup of $G$ or $\dim{\mca_1^G}\leq 5$ (see \cite{stab-dos}).  As an application of Theorem \ref{sc-intro}, we have found many unexpected new examples of standard local maxima, including the exceptional full flag manifolds of $E_6$, $E_7$, and $E_8$, which respectively have $\dim{\mca_1^G}= 35, 63, 119$ (see Tables \ref{tableIB2} and \ref{tableIB3}). 

\begin{corollary}\label{main2-cor1}
Among the $20$ spaces with exceptional $G$ in Theorem \ref{main2}, $g_{\kil}$ is a local maximum on $14$ of them and a local minimum on $3$ of them.  
\end{corollary}

In contrast, the standard metric is $G$-unstable on most of the spaces with $G$ classical.  Actually, $g_{\kil}$ is even a local minimum on many of them (see Table \ref{tableIA}).    

It is still an open problem whether there are only finitely many $G$-invariant Einstein metrics (up to homothety) on a given compact homogeneous space $M=G/K$.  This has been conjectured to hold in the multiplicity-free isotropy representation case in \cite{BhmWngZll}.  It follows from the compactness theorem 
\cite[Theorem 1.6]{BhmWngZll} that the conjecture is equivalent to the $G$-{\it rigidity} (i.e., it is an isolated point in the moduli space of all $G$-invariant unit volume Einstein metrics up to homothety) of any $G$-invariant Einstein metric (see \cite[Section 3.1]{stab-tres} for a more detailed treatment). 

Except for a few spaces, $g_{\kil}$ is $G$-non-degenerate in all the cases where $\Spec(\lic_\pg)$ was computed (i.e., in $23$ cases of a total of $32$, see also Table \ref{tableIAA}), which implies the following on $G$-rigidity. 

\begin{corollary}\label{main2-cor2}
The standard metric is an isolated point in the space of all $G$-invariant unit volume Einstein metrics on $M=G/K$ for all the spaces in Theorem \ref{main2} where $\Spec(\lic_\pg)$ is known, except possibly for the spaces $\SO(2n)/T^n$, $n\geq 4$, $E_6/\SO(3)^3$ and $E_8/\SO(9)$.     
\end{corollary}

We also obtained that the standard metric is $G$-degenerate on $\SU(4)/T^3$ (see Table \ref{tableIA}, 1a.2) and $E_6/\SU(2)\times\SO(6)$ (see Table \ref{tableIB1}, 4), but it is known to be $G$-rigid on these spaces (see \cite[pp.78]{Skn} and \cite[IV.16]{DckKrr}, respectively).  We do not know whether the standard metric is indeed $G$-degenerate or not on the spaces $E_6/\SO(3)^3$ (see Table \ref{tableIB1}, 2) and $E_8/\SO(9)$ (see Table \ref{tableIB2}, 9).  Concerning $\SO(2n)/T^n$, $n\geq 4$ (see Table \ref{tableIA}, 1b.2), on which the standard metric has nullity $n-1$, it was proved in \cite{locmax} that one of such null directions gives rise to an inflection point of $\scalar$, showing that it is not a local maximum.   

We next highlight some other general consequences of the theorem.   

\begin{enumerate}[{\small $\bullet$}]    
\item The Lichnerowicz Laplacian restricted to $\tca\tca_g^G$, has either only one or two eigenvalues in all the cases where $\Spec(\lic_\pg)$ was computed, with the only exception of the full flag manifolds $\SO(2n)/T^n$, $n\geq 4$, where $\lic_\pg$ has three different eigenvalues.  

\item Among the homogeneous spaces considered in Theorem \ref{main2}, there are exactly four cases where $K$ is a maximal subgroup of $G$:  
$$
\SO(n^2)/\SO(n)\times\SO(n), \quad \SO(4n^2)/\Spe(n)\times\Spe(n), \quad E_8/\SO(5), \quad E_8/\SU(5)\times\SU(5), 
$$
they all have two isotropy summands (see Table \ref{tableIA}, items 3a and 3b and Table \ref{tableIB2}, items 8 and 11, respectively).  For these spaces, it is well known that there exists a global maximum (see \cite{WngZll}) and at most two other $G$-invariant Einstein metrics (see \cite{DckKrr}).  We obtained that $g_{\kil}$ is indeed the global maximum and it is the only $G$-invariant Einstein metric for the first and fourth spaces and that $g_{\kil}$ is a local maximum on $E_8/\SO(5)$.  However, we were not able to figure out whether $g_{\kil}$ is a local maximum or not on $\SO(4n^2)/\Spe(n)\times\Spe(n)$ due to the huge length of computations needed.  What we know in this case is that either $g_{\kil}$ is a global maximum and the unique $G$-invariant Einstein metric or $g_{\kil}$ is a local minimum and there exist other two $G$-invariant Einstein metrics.    
\end{enumerate}

Concerning the rest of the proof of Theorem \ref{main2}, the nine spaces with two isotropy summands which are not covered by any criterion are case by case worked out in \S\ref{r2-sec}, by combining the three formulas for $\rho$ in terms of the structural constants, the Killing form constants and the Casimir constant, respectively.  

The most difficult case is by far the conceptual construction $\SO(\nn)/K$ listed in Table \ref{tableIA}, items 4 and 5, where the inclusion of $K$ is defined by the isotropy representation of a certain symmetric space (see \S\ref{som-sec}).  We compute the structural constants and detect eigenvectors of $\lic_\pg$ with eigenvalue $\frac{\nn}{2(\nn-2)}$, which is less than $2\rho$ and so $g_{\kil}$ is always $G$-unstable (see Theorem \ref{som-main} for more results on this class of spaces, including the computation of $\Spec(\lic_\pg)$ in several cases).  

The proof of Theorem \ref{main2} concludes in \S\ref{e7-sec} with the space $E_7/\SU(2)^7$, for which the uniform behavior of their structural constants allows to obtain the matrix of $\lic_\pg$ and consequently $\lambda_L^G$ explicitly.   

Finally, for completeness, we compute in \S\ref{rho-sec} the Einstein constants and the spectrum of $\lic_\pg$ for $10$ spaces which were solved without needing these data by the criteria (except Table \ref{tableIB1}, 4).  This information is provided in Table \ref{tableIAA} for the spaces in Table \ref{tableIA} and in the same tables for the exceptional cases.

\vs \noindent {\it Acknowledgements.}  The authors are very grateful with Christopher B\"ohm, McKenzie Wang and Wolfgang Ziller  for many helpful conversations.

\section{Preliminaries}\label{preli} 
We consider a compact connected differentiable manifold $M$ of dimension $d$ and assume that $M$ is homogeneous.  We also fix an almost-effective transitive action of a compact Lie group $G$ on $M$.  The $G$-action determines a presentation $M=G/K$ of $M$ as a homogeneous space, where $K\subset G$ is the isotropy subgroup at some point $o\in M$.  Let $\mca^G$ denote the space of all $G$-invariant Riemannian metrics on $M$, it is an open cone in $\sca^2(M)^G$, the finite-dimensional vector space of all $G$-invariant symmetric $2$-tensors.  Thus $\mca^G$ is a differentiable manifold of dimension between $1$ and $\tfrac{d(d+1)}{2}$.  

As well known, $g\in\mca_1^G$ is Einstein if and only if $g$ is a critical point of the scalar curvature functional $\scalar:\mca_1^G\longrightarrow \RR$, where $\mca^G_1\subset\mca^G$ is the codimension one submanifold of all unit volume metrics.  The stability type of a $G$-invariant Einstein metric on $M$ as a critical point of $\scalar|_{\mca_1^G}$ is therefore encoded in the signature of the second derivative or Hessian $\scalar''$.  We refer to \cite[Section 3]{stab-tres} and \cite[Section 2]{stab-dos} for more detailed treatments.

\begin{figure}
{\small 
$$
\xymatrix{
\underset{\text{Ricci flow}}{\text{dyn.stable}} 
&& \underset{{2\rho\notin\operatorname{Spec}(\Delta_L|_{\mathcal{T}\mathcal{T}_g})}}{\text{nondeg.}} \ar[r]& 
\text{rigid} 
& \underset{2\rho=\lambda_L}{\text{neut.stable}} \ar[d]\ar[ld]
\\ 
\ar[u]\underset{2\rho<\lambda_L \& 2\rho<\lambda_\Delta}{\nu\text{-stable}} \ar[r] & 
\underset{2\rho<\lambda_L}{\text{stable}} \ar[r]\ar[ur] \ar[d]& 
\underset{\text{of}\; \widetilde{\scalar}|_{\mathcal{C}_1}}{\text{\bf loc.max.}} 
\ar[r] \ar[d] &  \underset{2\rho\leq\lambda_L}{\text{semistable}} \ar[d]& \underset{{2\rho\in\operatorname{Spec}(\Delta_L|_{\mathcal{T}\mathcal{T}_g})}} {\text{deg.}}
\\ 
&\underset{2\rho<\lambda_L^G}{G\text{-stable}} \ar[r] \ar[rd]&
\underset{\text{of}\; \scalar|_{\mathcal M_1^G}}{\text{{\bf loc.max.}}} 
\ar[r]&
\underset{2\rho\leq\lambda_L^G}{G\text{-semistable}} &
\underset{2\rho\in\operatorname{Spec}(\Delta_L|_{\mathcal{T}\mathcal{T}_g^G})}{G\text{-deg.}} \ar[u] 
\\ 
&&
\underset{2\rho\notin\operatorname{Spec}(\Delta_L|_{\mathcal{T}\mathcal{T}_g^G})} {G\text{-nondeg.}} \ar[r]& 
G\text{-rigid} & 
\underset{2\rho=\lambda_L^G}{G\text{-neut.stable}} \ar[ul]\ar[u]
}
$$}

{\small	
$$
\xymatrix{
\underset{2\rho>\lambda_L^G}{G\text{-unstable}} \ar[r]&
\underset{2\rho>\lambda_L}{\text{unstable}} \ar[r] &
\underset{2\rho>\lambda_L \,\text{or}\, 2\rho>\lambda_\Delta}{\nu\text{-unstable}} \ar[r]&
\underset{\text{Ricci flow}}{\text{dyn.unstable}}
}
$$}
\caption{Stability and critical point types of an Einstein metric $g$ with $\operatorname{Rc}(g)=\rho g$.}\label{stabtypes}
\end{figure}
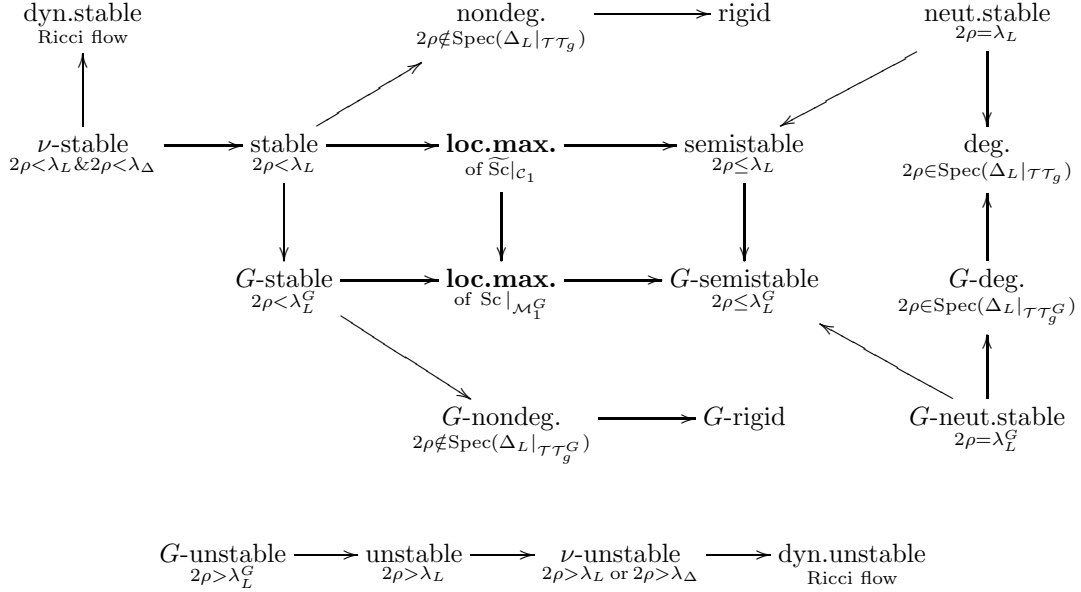

According to \cite[Section 3.4]{stab-tres}, the tangent space decomposes as 
$$
T_g\mca_1^G = T_gN_G(K)^*g \oplus \tca\tca_g^G, 
$$
where $N_G(K)\subset\Diff(M)$ is the normalizer of $K$ and $\tca\tca_g^G$ is the space of $G$-invariant TT-tensors.  Note that $T_gN_G(K)\cdot g$ is the space of trivial variations, that is, the tangent space of the $G$-equivariant isometry class of $g$.  This gives rise to the different notions of $G$-stability given in Definition \ref{stab-def-intro}.  We note that $G$-semistability must hold for any local maximum of $\scalar|_{\mca^G_1}$ and that a $G$-neutrally stable Einstein metric may or may not be a local maximum.   On the other hand, any $G$-unstable metric is a saddle point unless $\scalar''_g|_{\tca\tca_g^G}>0$, in which case $g$ is a local minimum of $\scalar|_{\mca^G_1}$.  The \emph{coindex} is the dimension of the maximal subspace of $\tca\tca_g^G$ on which $\scalar''_g$ is positive definite and the {\it nullity} is $\dim{\Ker\scalar''_g|_{\tca\tca_g^G}}$.   

Let $\ggo=\kg\oplus\pg$ be any reductive decomposition for $M=G/K$, where $\ggo$ and $\kg$ are the Lie algebras of $G$ and $K$, respectively.  This provides the usual identifications $T_oM\equiv\pg$ and  
$$
\sym(\pg)^K:=\{A\in\sym(\pg): [\Ad(K),A]=0\} \equiv \sca^2(M)^G,
$$
where $\sym(\pg):=\{A\in\glg(\pg):A^t=A\}$.  Furthermore, any $g\in\mca^G$ is determined by the $\Ad(K)$-invariant inner product $\ip:=g_o$ on $\pg$.  

The second variation of the total scalar curvature at any Einstein metric $g$ on $M$ (not necessarily homogeneous), say with $\ricci(g)=\rho g$, is given on $\tca\tca_g$ by
\begin{equation}\label{ScLL}
\scalar''_g = \unm(2\rho\id-\Delta_L),
\end{equation}
where $\Delta_L$ is the Lichnerowicz Laplacian of $g$ (see \cite[4.64]{Bss}).  For each $g\in\mca^G$, we consider the self-adjoint operator
$$
\lic_\pg=\lic_\pg(g):\sym(\pg)^K\longrightarrow\sym(\pg)^K,  
$$
such that  
$$
\Delta_LT = \la\lic_\pg A\cdot,\cdot\ra, \qquad\forall T\in\tca\tca_g^G, \quad T=\la A\cdot,\cdot\ra\in \sca^2(M)^G, \quad A\in\sym(\pg)^K,
$$
where $\la A,B\ra:=\tr{AB}$.  

According to \eqref{ScLL}, the $G$-stability type of $g$ is therefore determined by how is the constant $2\rho$ suited relative to the spectrum of $\lic_\pg$.  Let $\lambda_\pg=\lambda_{\pg}(g)$ and $\lambda_\pg^{\max}=\lambda_{\pg}^{\max}(g)$ denote, respectively, the minimum and maximum eigenvalues of  $\lic_\pg$ restricted to the subspace $\tca\tca_g^G$.  Recall that $\lambda_\pg$ was denoted by $\lambda_L^G$ in the introduction.  The characterizations of the $G$-stability types in terms of $\lambda_\pg$ and the Einstein constant $\rho$ were given in Definition \ref{stab-def-intro}.  We note that an Einstein metric $g\in\mca_1^G$ is a local minimum of $\scalar|_{\mca_1^G}$ if $\lambda_\pg^{\max}<2\rho$ and that $\lic_\pg|_{T_gN_G(K)^*g}=2\rho\id$, since $\scalar''_g$ vanishes on trivial variations (cf.\ \cite[(30)]{stab-tres}).  

Assume from now on that the metric $g\in\mca^G$ is \emph{naturally reductive} with respect to $G$ and $\pg$, i.e.,  the map $\ad_\pg{X}:\pg\rightarrow\pg$, $Y\mapsto [X,Y]_\pg$ is skew-symmetric for any $X\in\pg$, where $[X,Y]_\pg$ denotes the projection onto $\pg$ of the Lie bracket $[X,Y]$ relative to $\ggo=\kg\oplus\pg$.  In that case, 
$$
T_g\mca_1^G =\tca\tca_g^G=\sym_0(\pg)^K:=\{ A\in\sym(\pg)^K:\tr{A}=0\},
$$ 
and the Lichnerowicz Laplacian $\lic_\pg$ is given by 
\begin{equation}\label{Lpnr-intro}
\lic_\pg A:=-\unm\sum[\ad_\pg{X_i},[\ad_\pg{X_i},A]], \qquad\forall A\in\sym(\pg)^K,
\end{equation}
where $\{ X_i\}$ is any $g$-orthonormal basis of $\pg$ (see \cite[Section 5]{stab-tres}).  

\begin{example}
If $g_{\kil}$ is the Killing left-invariant metric on any compact simple Lie group $G$, which always satisfies $\ricci(g_{\kil})=\unc g_{\kil}$, then 
$
\lic_\pg(g_{\kil})=\unm \cas_\tau,
$ 
where $\cas_\tau$ is the Casimir acting on the representation $\sym(\ggo)$ of $\ggo$ relative to the Killing form.  Thus the $G$-stability type of $g_{\kil}$ follows from the spectrum of $\cas_\tau$ (see \cite[Table 1]{stab-tres}), which is well known and can be computed using representation theory (see Table \ref{table1}).  Thus $g_{\kil}$ is $G$-stable, except for $\SU(n)$, $n\geq 3$ and $\Spe(n)$, $n\geq 2$, where it is $G$-neutrally stable and $G$-unstable, respectively.  It is proved in \cite{locmax} that on $\SU(n)$, $n\geq 3$ the metric $g_{\kil}$ is not a local maximum of $\scalar|_{\mca_1^G}$ (see also \cite[Theorem 3, 4)]{Nkn}).      
\end{example}

Any orthogonal decomposition $\pg=\pg_1\oplus\dots\oplus\pg_r$ in $\Ad(K)$-invariant subspaces $\pg_1,\dots,\pg_r$ ($d_i:=\dim{\pg_i}$) determines structural constants given by,
\begin{equation}\label{SC}
[ijk]:=\sum_{\alpha,\beta,\gamma} \la[X_\alpha^i,X_\beta^j], X_\gamma^k\ra^2,
\end{equation}
where $\{ X_\alpha^i\}$ is an orthonormal basis of $\pg_i$.  Note that the number $[ijk]$ is invariant under any permutation of $ijk$ by natural reductivity.  Consider the orthonormal subset $\left\{ I_1,\dots, I_r\right\}$ of $\sym(\pg)^K$ given by $I_k|_{\pg_i}:=\delta_{ki}\tfrac{1}{\sqrt{d_k}}I_{\pg_k}$, where $I_\vg$ will always denote the identity map on the vector space $\vg$.

According to \cite[Theorem 5.3]{stab-tres} (see also \cite[Theorem 3.1]{stab-dos}),  
\begin{equation}\label{Lpijk}
\begin{aligned}
\la\lic_\pg I_k, I_k\ra &= \tfrac{1}{d_k}\sum\limits_{j\ne k;i} [ijk], \qquad\forall k, \\ 
\la\lic_\pg I_k, I_m\ra &=  -\tfrac{1}{\sqrt{d_k}\sqrt{d_m}}\sum\limits_{i} [ikm], \qquad\forall k\ne m.  
\end{aligned}
\end{equation}
In the case when $G/K$ is \emph{multiplicity-free}, i.e., the subspaces $\pg_k$'s are $\Ad(K)$-irreducible and pairwise inequivalent, $\left\{ I_1,\dots, I_r\right\}$ is an orthonormal basis of $\sym(\pg)^K$ and so the above numbers are precisely the entries of the matrix of $\lic_\pg$.  
This provides a useful tool to compute the whole spectrum of $\lic_\pg$, in order to establish the $G$-stability type of the Einstein metric $g$.  Note that the vector $(\sqrt{d_1},\dots,\sqrt{d_r})$ is always in the kernel of $\lic_\pg$ since these are precisely the coordinates of the identity map $I_\pg\in\sym(\pg)^K$ (see \eqref{Lpnr-intro}).     

In the general case, i.e., the subspaces $\pg_k$'s are only assumed to be $\Ad(K)$-invariant, the spectrum of the symmetric $r\times r$ matrix defined as in \eqref{Lpijk} (restricted to the hyperplane $\sum d_ia_i=0$ and intersected with $\tca\tca_g^G$ if necessary) is still contained in $[\lambda_\pg, \lambda_\pg^{\max}]$.  In particular, the $G$-instability of $g$ (say with $\ricci(g)=\rho g$) follows as soon as some eigenvalue of such matrix is less than $2\rho$ (see \cite[Remark 3.3]{stab-dos}).

\section{Standard Einstein metrics}\label{standard-sec}

Any homogeneous space $M^d=G/K$ with $G$ compact semisimple admits a canonical $G$-invariant metric $g_{\kil}$, called the \emph{standard} metric and given by $\ip=-\kil_\ggo|_{\pg\times\pg}$, where $\kil_\ggo$ is the Killing form of $\ggo$ and $\ggo=\kg\oplus\pg$ is the $\kil_\ggo$-orthogonal decomposition.  Note that $g_{\kil}$ is clearly naturally reductive with respect to $G$ and $\pg$.  We consider any orthogonal decomposition $\pg=\pg_1\oplus\dots\oplus\pg_r$ in $\Ad(K)$-irreducible subspaces and denote by $\chi$ the isotropy representation of $K$ and $\kg$ on $\pg$.  We refer to \cite[Chapter 1]{WngZll2} and \cite[Chapter 7, Section G]{Bss} for more detailed treatments on standard metrics.

\subsection{Einstein equation} 
The Ricci operator of the standard metric is given by 
\begin{equation}\label{ricgB}
\Ricci(g_{\kil}) = \Mm_{\lb_\pg} + \unm I_\pg = \unc I_\pg + \unm\cas_\chi, 
\end{equation}
where $\Mm_{\lb_\pg}=\unc\sum (\ad_\pg{X_i})^2$ and $\cas_\chi=-\sum (\ad{Z_j}|_\pg)^2$ is the Casimir acting on the isotropy representation $\chi$ with respect to $-\kil_\ggo|_{\kg}$ (see \cite[(16), (31)]{stab-tres} and \cite[Corollary 1.7]{WngZll2}, respectively).  Here $\{ Z_j\}$ and $\{ X_i\}$ are respectively $-\kil_\ggo$-orthonormal basis of $\kg$ and $\pg$.  Note that $\Mm_{\lb_\pg}\leq 0$ and $\cas_\chi\geq 0$.  The following conditions are therefore equivalent: 

\begin{enumerate}[{\small $\bullet$}] 
\item $g_{\kil}$ is Einstein, say $\ricci(g_{\kil}) = \rho g_{\kil}$ (i.e., $\Ricci(g_{\kil}) = \rho I_\pg$).  

\item $\Mm_{\lb_\pg}=(\rho-\unm)I_\pg$ (since $\tr{\Mm_{\lb_\pg}}=-\unc|\lb_\pg|^2$, it follows that $\rho\leq\unm$, where equality holds if and only if $[\pg,\pg]\subset\kg$, i.e., $(G/K,g_{\kil})$ is a locally symmetric space).  

\item $\cas_\chi=(2\rho-\unm)I_\pg$ (in particular, $\unc\leq\rho$, where equality holds if and only if $\kg=0$).  

\item $\cas_{\chi_k}=(2\rho-\unm)I_{\pg_k}$ for every $k=1,\dots,r$, where $\cas_{\chi_k}$ is the Casimir operator with respect to $-\kil_\ggo|_\kg$ acting on the irreducible representation $\pg_k$ of $\kg$ (note that $\cas_{\chi_k}$ is always a multiple of the identity).  
\end{enumerate}

We write the adjoint map by
\begin{equation}\label{adXnr}
\ad{X} = \left[\begin{matrix} 0&a(X)\\ -a(X)^{t}&\ad_\pg{X} \end{matrix}\right], \qquad\forall X\in\pg, 
\end{equation}
where $a(X)^{t}:\kg\rightarrow\pg$ is the adjoint operator of $a(X):\pg\rightarrow\kg$ with respect to $-\kil_\ggo$, i.e., $\la a(X)^{t}Z,Y\ra=-\kil_\ggo(Z,[X,Y])$ for all $Z\in\kg$, $X,Y\in\pg$. 

Most of the following lemma is contained in \cite{WngZll2} (see also \cite[Chapter 7, Section G]{Bss}); we include a proof for completeness.  

\begin{lemma}\label{formulas}
The following formulas hold:
\begin{enumerate}[{\rm (i)}] 
\item $\kil_\chi=-\sum a(X_i)a(X_i)^t$, where $\kil_\chi:\kg\rightarrow\kg$ is defined by $-\kil_\ggo(\kil_\chi Z,Z)=\tr{\chi(Z)^2}$ for all $Z\in\kg$.  

\item $\cas_\chi=\sum a(X_i)^ta(X_i)$.  In particular, $\tr{\cas_\chi}=-\tr{\kil_\chi}$.  

\item $\cas_{\ad_\kg}=I_\kg+\kil_\chi$, where $\cas_{\ad_\kg}$ is the Casimir acting on the adjoint representation of $\kg$ with respect to $-\kil_\ggo|_{\kg}$.   

\item $\sum a(X_i)\ad_\pg{X_i}=0$.  

\item $2\cas_\chi-4\Mm_{\lb_\pg}=I_\pg$.  
\end{enumerate}
\end{lemma}

\begin{remark}
Using part (ii), each of the formulas for the Ricci operator given in \eqref{ricgB} follows from the other one.  It can be easily shown that $A=\la \cas_\chi\cdot,\cdot\ra$ is precisely the symmetric bilinear form defined in \cite[(1.1)]{WngZll2}.  
\end{remark}

\begin{proof}
Part (i) follows from
\begin{align*}
\tr{\chi(Z)^2} =& \sum \la[Z,[Z,X_i]],X_i\ra = -\sum \la[X_i,Z],[X_i,Z]\ra \\ 
=& \sum -\kil_\ggo((\ad{X_i})^2Z,Z) = \sum-\kil_\ggo(- a(X_i)a(X_i)^tZ,Z),
\end{align*}
and part (ii) from 
\begin{align*}
\la\cas_\chi X,X\ra =& -\sum \la(\ad{Z_j})^2X,X\ra = \sum \la[Z_j,X],X_i\ra^2 \\ 
=& \sum -\kil_\ggo([X_i,X]_\kg,[X_i,X]_\kg) = \sum\la a(X_i)^ta(X_i)X,X\ra.  
\end{align*}
On the other hand, it is easy to see that the Casimir acting on the adjoint representation of $\ggo$ with respect to $-\kil_\ggo$, $\cas_{\ad}=-\sum (\ad{Z_j})^2-\sum (\ad{X_i})^2$, is given by 
$$
\cas_{\ad} = \left[\begin{matrix} \cas_{\ad_\kg}-\kil_\chi & -\sum a(X_i)\ad_\pg{X_i}\\ 
\sum \ad_\pg{X_i}a(X_i)^t & 2\cas_\chi-4\Mm_{\lb_\pg} \end{matrix}\right].   
$$
The remaining parts therefore follow from the fact that $\cas_{\ad}=I_\ggo$.  
\end{proof}

\subsection{Einstein constant}
We take a decomposition $\kg=\kg_1\oplus\dots\oplus\kg_s$ in ideals of $\kg$ and assume that $\kil_{\kg_i}=c_i\kil_\ggo|_{\kg_i}$ for some $c_i\in\RR$, for each $i=1,\dots,s$.  Note that $0\leq c_i\leq 1$, $c_i=0$ if and only if $\kg_i$ is abelian and $c_i=1$ if and only if $\kg_i$ is an ideal of $\ggo$ (see \cite[Theorem 11, pp.35]{DtrZll} for more information on these constants $c_i$'s).  This assumption in particular holds if each $\kg_i$ is either simple or abelian.   It is easy to see that $\cas_{\ad_\kg}$ is the block map $[c_1I_{\kg_1},\dots,c_sI_{\kg_s}]$ and so by Lemma \ref{formulas}, (iii),  
\begin{equation}\label{Bchi}
\kil_\chi=[(c_1-1)I_{\kg_1},\dots,(c_s-1)I_{\kg_s}]. 
\end{equation}  

If $g_{\kil}$ is Einstein, then 
\begin{equation}\label{ci}
\sum_{i=1}^s(1-c_i)\dim{\kg_i} = -\tr{\kil_\chi} = \tr{\cas_\chi} = (2\rho-\unm)d,
\end{equation}
from which the following useful formula for the Einstein constant follows (cf.\ \cite[pp.34]{DtrZll}, \cite[pp.568]{WngZll2} and \cite[(7.94)]{Bss}, where the following typo appears: the sum should be from $i=0$ to $r$):  
\begin{equation}\label{rhoci}
\rho=\unc + \tfrac{1}{2d}\sum_{i=1}^s(1-c_i)\dim{\kg_i}.  
\end{equation}
In particular, if $\kg$ is abelian then $\rho=\unc+\tfrac{1}{2d}\dim{\kg}$ and, if $[\pg,\pg]\subset\kg$ (i.e.\ $M=G/K$ is a symmetric space) then $c_i=\tfrac{2\dim{\kg}-d}{2\dim{\kg}}$ by \cite[Theorem 11, (i)]{DtrZll} and so $\rho=\unm$.

A well-known formula for the Ricci eigenvalues of $g_{\kil}$ in terms of the structural constants defined in \eqref{SC} (see e.g.\ \cite[(34)]{stab-tres} or \cite[(18)]{stab-dos}) is given by 
\begin{equation}\label{rhoijk}
\rho_k = \unm - \tfrac{1}{4d_k}\sum_{i,j} [ijk],\qquad\forall k=1,\dots,r,
\end{equation}
and it follows from \cite[Lemma (1.5)]{WngZll} that 
\begin{equation}\label{sumijk}
\sum_{i,j} [ijk] = d_k(1-2a_k), \qquad\forall k=1,\dots,r, 
\end{equation}
where $\cas_{\chi_k}=a_kI_{\pg_k}$.  Recall that $g_{\kil}$ is Einstein if and only if $a_1=\dots =a_r$ (in that case, $\rho=\unc+\unm a_1$), and that $d=d_1+\dots+d_r=\dim{\pg}=\dim{\ggo}-\dim{\kg}$.

\begin{table}
{\small 
$$
\begin{array}{c|c|c|c|c|c}
\text{No.} & \ggo/\kg & \text{Condition} & r  & G\text{-stab.\ type} & \text{Ref.} 
\\[2mm] \hline \hline \rule{0pt}{14pt}
1a.1  & \tfrac{\sug(3)}{\sg(3\cdot\ug(1))}  & - & 3   & G\text{-unst., loc.min.} &\text{\cite{stab-tres}, \cite{Nkn}}
\\[2mm]  \hline \rule{0pt}{14pt}
1a.2 & \tfrac{\sug(4)}{\sg(4\cdot\ug(1))}  & - & 6   & 
{
	\underset{\text{coindex}=3}{G\text{-unst.}}, \underset{\text{nullity}=2}{G\text{-deg.}} 
}
&{\text{\cite{stab-tres}}}
\\[2mm]  \hline \rule{0pt}{14pt}
1a.3 & \tfrac{\sug(n)}{\sg(n\cdot\ug(1))} & n\geq 5 & \tfrac{n(n-1)}{2}   & 
{ 
	\underset{\text{coindex}=n-1}{G\text{-unst.}} 
}, \text{saddle}
&{\text{\cite{stab-tres}}}
\\[2mm]  \hline \rule{0pt}{14pt}
1b.1 & \tfrac{\sog(6)}{3\cdot\sog(2)}  & - & 6  & 
{ 
	\underset{\text{coindex}=3}{G\text{-unst.}}, \underset{\text{nullity}=2}{G\text{-deg.}}  
}, \text{saddle}
&\text{\S\ref{so2n-sec}, \cite{locmax}}
\\[2mm]  \hline \rule{0pt}{14pt}
1b.2 & \tfrac{\sog(2n)}{n\cdot\sog(2)} & n\geq 4 & n(n-1)   & 
{ 
	\underset{\text{nullity}=n-1}{G\text{-neut.stab.}} 
}, \text{saddle}
&\text{\S\ref{so2n-sec}, \cite{locmax}}
\\[2mm]  \hline \rule{0pt}{14pt}
2a & \tfrac{\sug(nk)}{\sg(n\cdot\ug(k))}  & k\geq 2, n\geq 3 & \tfrac{n(n-1)}{2}   & G\text{-unst.} &{\text{\cite{stab-tres}}}
\\[2mm]  \hline \rule{0pt}{14pt}
2b & \tfrac{\spg(nk)}{n\cdot\spg(k)}  & k\geq 1, n\geq 3 & \tfrac{n(n-1)}{2}  & G\text{-unst.} &{\text{\cite{stab-tres}}}
\\[2mm]  \hline \rule{0pt}{14pt}
2c & \tfrac{\sog(nk)}{n\cdot\sog(k)}  & k\geq 3, n\geq 3 & \tfrac{n(n-1)}{2}   & G\text{-unst.} &{\text{\cite{stab-tres}}}
\\[2mm]  \hline \rule{0pt}{14pt}
 3a  & \tfrac{\sog(n^2)}{\sog(n)\oplus\sog(n)} & n\geq 3 & 2  & G\text{-stab., glob.max.} & { \S\ref{A.3a}}
\\[2mm]  \hline \rule{0pt}{14pt}
 3b  & \tfrac{\sog(4n^2)}{\spg(n)\oplus\spg(n)} & n\geq 2 & 2  &  &{ \S\ref{A.3b}}
\\[2mm]  \hline \rule{0pt}{14pt}
 4 & 
 \tfrac{\sog(n)}{\kg} 
   & n=\dim{\kg} & 
 	\tfrac{l(l+1)}{2} & 

 	\underset{\text{coindex}=l-1}{G\text{-unst.}} 
, \text{saddle}
 &\S\ref{som-sec}
\\[2mm]  \hline \rule{0pt}{14pt}
5 & \tfrac{\sog(\nn)}{\kg_1\oplus\dots\oplus\kg_l} & l\geq 2 & 
	\text{see \eqref{eq-conceptual:r}}
&
 	G\text{-unst.}
&\S\ref{som-sec}
\\[2mm]  \hline \rule{0pt}{14pt}
6 & \tfrac{\sug(pq+l)}{\sug(p)\oplus\sug(q)\oplus\ug(l)} & p,q\geq 2, l\geq 3 & 2 & G\text{-unst., loc.min.} & \S\ref{A.6}
\\[2mm]  \hline \rule{0pt}{14pt}
7a & \tfrac{\spg(3n-1)}{\spg(n)\oplus\ug(2n-1)} & n\geq 1 & 2  & G\text{-unst., loc.min.} &\text{Cor.}\ref{sc1},\S\ref{A.7a}, \text{\cite{Nkn}}
\\[2mm]  \hline \rule{0pt}{14pt}
7b & \tfrac{\sog(3n+2)}{\sog(n)\oplus\ug(n+1)} & n\geq 3 & 2 & G\text{-unst., loc.min.} & \text{Cor.}\ref{sc1}, \text{\cite{Nkn}}
\\[2mm]  \hline \rule{0pt}{14pt}
8 & \tfrac{\sog(26)}{\spg(1)\oplus\spg(5)\oplus\sog(6)} & - & 2  & G\text{-unst., loc.min.}  &\S\ref{A.8} 
\\[2mm]  \hline \rule{0pt}{14pt}
9 & \tfrac{\sog(8)}{\ggo_2} & - & 2 \; \text{(nmf)} & G\text{-unst., loc.min.} &\S\ref{A.9}, \text{\cite{Gtr}, \cite{Nkn}}
%
\\[2mm] \hline\hline
\end{array}
$$}
\caption{\cite[Table IA, pp.577]{WngZll2}.  See \cite[Table 2]{stab-tres} for more stability information on the cases 2a, 2b and 2c. In case $6$, $pql=p^2+q^2+1$.  
See Theorem~\ref{som-main} for information about the coindex in case 5. 
}\label{tableIA}
\end{table}

\begin{table}
{\small 
$$
\begin{array}{c|c|c|c|c|c|c|c|c}
\text{No.} & \ggo/\kg & r & \rho & \lambda_\pg  & G\text{-stab.\ type} & \text{C1} & \text{C2} & \text{Ref.}
\\[2mm] \hline \hline \rule{0pt}{14pt}
1 & \tfrac{\fg_4}{\sping(8)} & 3 & \tfrac{4}{9} & \tfrac{1}{3} & G\text{-unst., loc.min.} & \checkmark & \checkmark &\text{\cite{stab-dos}}, \text{Cor.}\ref{sc2-cor}, \text{\cite{Nkn}} 
\\[2mm]  \hline \rule{0pt}{14pt}
2 & \tfrac{\eg_6}{\sog(3)\oplus\sog(3)\oplus\sog(3)} & 5\; \text{(nmf)} & \tfrac{5}{16}^* & &G\text{-semistab.}  & \checkmark  & \checkmark & \text{Cor.}\ref{sc2-cor} 
\\[2mm]  \hline \rule{0pt}{14pt}
3 & \tfrac{\eg_6}{\sping(8)\oplus\RR^2}  & 3 & \tfrac{5}{12} &\unm  & G\text{-unst., loc.min.} & \checkmark &\text{No}&\text{\cite{stab-dos}}, \text{Cor.}\ref{sc1}, \text{\cite{Nkn}} 
\\[2mm]  \hline \rule{0pt}{14pt}
4 & \tfrac{\eg_6}{\sug(2)\oplus\sog(6)}  & 2 & \tfrac{3}{8}^* & \tfrac{3}{4}^* & G\text{-neut.stab., saddle} & \text{No} & \text{No}&\S\ref{B.4}
\\[2mm]  \hline \rule{0pt}{14pt}
5 & \tfrac{\eg_7}{\sog(8)} & 3 & \tfrac{13}{36} & \tfrac{5}{6} & G\text{-stab., loc.max.} & \text{No} &\text{No}&\text{\cite{stab-dos}}
\\[2mm]  \hline \rule{0pt}{14pt}
6 & \tfrac{\eg_7}{\sping(8)\oplus 3\cdot\sug(2)} & 3 & \tfrac{7}{18} &\tfrac{2}{3}  & G\text{-unst., loc.min.} & \text{No} & \text{No}&\text{\cite{stab-dos}}
\\[2mm] \hline\hline
\end{array}
$$}
\caption{\cite[Table IB, pp.578]{WngZll2}.  Excepting the item 2, in all cases $\lambda_\pg$ is the only nonzero eigenvalue of $\lic_\pg$ and has multiplicity $r-1$.} \label{tableIB1}
\end{table}

\begin{table}
{\small 
$$
\begin{array}{c|c|c|c|c|c|c|c|c}
\text{No.} & \ggo/\kg & r & \rho & \lambda_\pg & G\text{-stab.\ type}  & \text{C1} & \text{C2} & \text{Ref.}
\\[2mm] \hline \hline \rule{0pt}{14pt}
7 & \tfrac{\eg_7}{7\cdot\sug(2)} & 7 & \tfrac{1}{3}^* &\tfrac{7}{9}^* & G\text{-stab., loc.max.} & \text{No} &\text{No}& \S\ref{e7-sec}
\\[2mm]  \hline \rule{0pt}{14pt}
8 & \tfrac{\eg_8}{\sog(5)}  & 2 & && G\text{-stab., loc.max.} &\checkmark & \checkmark &\text{Cor.}\ref{sc2-cor}
\\[2mm]  \hline \rule{0pt}{14pt}
9 & \tfrac{\eg_8}{\sog(9)} & 3\; \text{(nmf)} & \tfrac{13}{40}^* && G\text{-semistab.} & \checkmark  & \checkmark & \text{Cor.}\ref{sc2-cor}
\\[2mm]  \hline \rule{0pt}{14pt}
10 & \tfrac{\eg_8}{\sping(9)}  & 2 & \tfrac{13}{40}^* & \tfrac{53}{60}^* & G\text{-stab., loc.max.} & \checkmark^*  & \checkmark^*&\text{Cor.}\ref{sc2-cor}, \S\ref{B.10}
\\[2mm]  \hline \rule{0pt}{14pt}
11 & \tfrac{\eg_8}{\sug(5)\oplus\sug(5)} & 2 & \tfrac{7}{20}^* &\tfrac{4}{5}^* & G\text{-stab., glob.max.} & \text{No} &\text{No}& \S\ref{B.11}
\\[2mm]  \hline \rule{0pt}{14pt}
12 & \tfrac{\eg_8}{4\cdot\sug(3)}  & 4 & \tfrac{19}{60}^* & \tfrac{4}{5}^* & G\text{-stab., loc.max.} & \checkmark & \checkmark&\text{Cor.}\ref{sc2-cor}
\\[2mm]  \hline \rule{0pt}{14pt}
13 & \tfrac{\eg_8}{4\cdot\sog(3)}  & 9\; \text{(nmf)} & \tfrac{11}{40}^* && G\text{-stab., loc.max.} &\checkmark& \checkmark&\text{Cor.}\ref{sc2-cor}
\\[2mm]  \hline \rule{0pt}{14pt}
14 & \tfrac{\eg_8}{\sping(8)\oplus\sping(8)}  & 3 & \tfrac{11}{30} & \tfrac{4}{5} & G\text{-stab., loc.max.} & \text{No} & \text{No}&\text{\cite{stab-dos}}
\\[2mm] \hline\hline
\end{array}
$$}
\caption{\cite[Table IB, pp.579]{WngZll2}.  In all the cases where it is provided, $\lambda_\pg$ is the only nonzero eigenvalue of $\lic_\pg$ and has multiplicity $r-1$.  $\checkmark^*$ means that the criteria only imply that $2\rho\leq\lambda_\pg$.} \label{tableIB2}
\end{table}

\begin{table}
{\small 
$$
\begin{array}{c|c|c|c|c|c|c|c|c|c}
\text{No.} & \ggo/\kg  & r & \rho & \lambda_\pg & \lambda_\pg^{\max}& G\text{-stab.\ type}  & \text{C1} & \text{C2}& \text{Ref.}
\\[2mm] \hline \hline \rule{0pt}{14pt}
15 & \tfrac{\eg_8}{8\cdot\sug(2)} & 14 & \tfrac{3}{10}^* &  \underset{\text{mult=}7}{\tfrac{4}{5}^*} 
 &  \underset{\text{mult=}6}{\tfrac{14}{15}^*}
& G\text{-stab., loc.max.} &\checkmark & \checkmark&\text{Cor.}\ref{sc2-cor}
\\[2mm]  \hline \rule{0pt}{14pt}
16 & \tfrac{\eg_8}{\sog(5)\oplus\sog(5)}  & 6\; \text{(nmf)} &\tfrac{7}{24}^* &&& G\text{-stab., loc.max.} & \checkmark & \checkmark&\text{Cor.}\ref{sc2-cor}
\\[2mm]  \hline \rule{0pt}{14pt}
17 & \tfrac{\eg_8}{\sug(3)\oplus\sug(3)} & 5\; \text{(nmf)} & \tfrac{17}{60}^*&&& G\text{-stab., loc.max.} &\checkmark & \checkmark&\text{Cor.}\ref{sc2-cor}
\\[2mm]  \hline \rule{0pt}{14pt}
18a & \tfrac{\eg_6}{\tg} & 36 & \tfrac{7}{24}  & \underset{\text{mult=}20}{\tfrac{3}{4}^*} 
 &  \underset{\text{mult=}15}{1^*} 
& G\text{-stab., loc.max.}&\checkmark &\checkmark&\text{Cor.}\ref{sc2-cor}
\\[2mm]  \hline \rule{0pt}{14pt}
18b & \tfrac{\eg_7}{\tg} & 63 & \tfrac{5}{18} &  \underset{\text{mult=}27}{\tfrac{7}{9}^*}
&  \underset{\text{mult=}35}{1^*}
& G\text{-stab., loc.max.}  &\checkmark &\checkmark&\text{Cor.}\ref{sc2-cor}
\\[2mm]  \hline \rule{0pt}{14pt}
18c & \tfrac{\eg_8}{\tg} & 120 & \tfrac{4}{15}  &  \underset{\text{mult=}35}{\tfrac{4}{5}^*}
& \underset{\text{mult=}84}{1^*}
&G\text{-stab., loc.max.}  &\checkmark &\checkmark&\text{Cor.}\ref{sc2-cor}
\\[2mm] \hline\hline
\end{array}
$$}
\caption{\cite[Table IB, pp.580]{WngZll2}.} \label{tableIB3}
\end{table}

\subsection{$G$-stability}\label{main-sec}
A complete classification of standard metrics which are Einstein was obtained by Wang and Ziller in \cite{WngZll2} in the case when $G$ is simple.  Besides isotropy irreducible spaces and assuming that $G$ and $K$ are both connected, the list consists of $12$ infinite families (two of them are actually conceptual constructions) and $22$ isolated examples, which have been collected in  Tables \ref{tableIA}, \ref{tableIB1}, \ref{tableIB2} and \ref{tableIB3} by using the same numbering as in \cite[Tables IA and IB]{WngZll2}.  

The following is the main result of this paper, providing a complete picture of the $G$-stability of these standard Einstein metrics.   

\begin{theorem}\label{main}
The $G$-stability types of all standard Einstein metrics on non-isotropy irreducible homogeneous spaces $M=G/K$ with $G$ simple are given as in Tables \ref{tableIA}, \ref{tableIB1}, \ref{tableIB2} and \ref{tableIB3}, except for the space $\SO(4n^2)/\Spe(n)\times\Spe(n)$.  
\end{theorem}

The proof of the theorem is worked out in Sections \ref{Gstab-sec} through \ref{e7-sec}.  

\begin{remark}
The following remarks are in order:  
\begin{enumerate}[(1)] 
\item The only cases for which the $G$-stability type had been known before are the flag manifolds listed in Table \ref{tableIA}, items 1a.1, 1a.2, 1a.3, 2a, 2b, 2c (see \cite[Table 2]{stab-tres}) and the generalized Wallach spaces given in Table \ref{tableIB1}, items 1, 3, 5, 6 and Table \ref{tableIB2}, item 14 (see \cite[Table 1]{stab-dos}).  

\item The number $r$ of irreducible isotropy summands is provided in the tables.  Recall that $\dim{\mca_1^G}\geq r-1$, where equality holds if and only if the space is multiplicity-free.  There are only six cases which are not multiplicity-free; they have been marked with `(nmf)'  on the right of the number $r$.  

\item Additional information on what kind of critical point of $\scalar|_{\mca_1^G}$ is the standard metric $g_{\kil}$ is also given, including the coindex (if positive and $<r-1$) and the nullity (if positive).  

\item The last column on the right of the tables provides a reference to the place in the paper (or to another paper) where the $G$-stability type was proved.  

\item Einstein constants (see Table \ref{tableIAA} for the spaces in Table \ref{tableIA}) appearing with an asterisk were computed in this paper; the authors were not able to find any of them in the literature.  

\item The eigenvalues $\lambda_\pg$ and $\lambda_\pg^{\max}$ of the Lichnerowicz Laplacian computed here were also marked with an asterisk and the corresponding multiplicity was added below the value in the cases when it is different from $r-1$ (see Table \ref{tableIAA} for the spaces in Table \ref{tableIA}).  Recall that $\tca\tca_g^G=\sym_0(\pg)^K$ for any standard metric $g$.      

\item All the eigenvalues of $\lic_\pg|_{\sym_0(\pg)^K}$ we found are nonzero, which implies that the standard metric is Ricci locally invertible (see \cite{PRP}) in all the cases where the whole spectrum of $\lic_\pg$ was computed.   

\item The meaning of the columns C1 and C2 is explained at the end of \S\ref{Gstab-sec}.  

\item We do not know whether each of the two $G$-semistable cases $E_6/\SO(3)^3$ (see Table \ref{tableIB1}, 2) and $E_8/\SO(9)$ (see Table \ref{tableIB2}, 9) is either $G$-stable or $G$-neutrally stable.   

\item In the case $E_8/\SO(5)$ (see Table \ref{tableIB2}, 8), $g_{\kil}$ is a local maximum as it is $G$-stable.  On the other hand, since $K$ is a maximal subgroup of $G$, we know that there exists a global maximum (see \cite[Theorem (2.2)]{WngZll}), but we were not able to figure out whether $g_{\kil}$ is the global maximum or not without more information on the structural constants.   
\end{enumerate}
\end{remark}

\section{Stability criteria for standard Einstein metrics}\label{Gstab-sec}

In this section, we give sufficient conditions for the $G$-stability types of $g_{\kil}$ on $G/K$ to hold involving only $\dim{G}$, $\dim{K}$ and the Casimir eigenvalues $\lambda_\tau$ and $\lambda_\tau^{\max}$ attached to $\ggo$ (see Table \ref{table1}).  Several cases, specially when $\ggo$ is exceptional, follow from these purely Lie theoretical criteria.  

Let $\cas_\tau$ denote the Casimir operator acting on the representation $\sym(\ggo)$ of $\ggo$ given by $\tau(X)A:=[\ad{X},A]$ with respect to $-\kil_\ggo$, i.e., 
$$
\cas_\tau=-\sum \tau(Z_j)^2-\sum\tau(X_i)^2.  
$$
The eigenvalues $\lambda_\tau<\lambda_\tau^{\midop}<\lambda_\tau^{\max}$ of $\cas_\tau$ restricted to $\sym_0(\ggo)$ are given in Table \ref{table1} for each compact simple Lie algebra $\ggo$ (it is only one eigenvalue for $\sug(2)$ and only two, say $\lambda_\tau<\lambda_\tau^{\max}$, for any exceptional $\ggo$).  The following result paves the way to estimate the eigenvalues $\lambda_\pg$ and $\lambda_\pg^{\max}$ of the Lichnerowicz Laplacian $\lic_\pg=\lic_\pg(g_{\kil})$ in terms of $\lambda_\tau$ and $\lambda_\tau^{\max}$.  

\begin{lemma}\label{castau}
$\cas_\tau$ leaves the subspace $\sym(\ggo)^K$ invariant and 
\begin{align*}
\cas_\tau\left[\begin{matrix} B&0\\ 0&0 \end{matrix}\right] &= \left[\begin{matrix} 
-2\kil_\chi B & 0\\ 
0 & -2\sum a(X_i)^tBa(X_i) 
\end{matrix}\right], \qquad\forall B\in\sym(\kg)^K,
\\
\cas_\tau\left[\begin{matrix} 0&0\\ 0&A \end{matrix}\right] &= \left[\begin{matrix} 
-2\sum a(X_i)Aa(X_i)^t & 2\sum a(X_i)A\ad_\pg(X_i) \\ 
-2\sum \ad_\pg(X_i)Aa(X_i)^t & \cas_\chi A +A\cas_\chi + 2\lic_\pg A\end{matrix}\right], \qquad \forall A\in\sym(\pg)^K,
\end{align*}
where $\lic_\pg$ is the Lichnerowicz Laplacian of $g_{\kil}$.  
\end{lemma}

\begin{remark}
Most of the times, the $K$-irreducible subspaces of $\kg$ are pairwise non-equiv\-a\-lent as $K$-representations to those of $\pg$, which implies that 
$$
\sym(\ggo)^K = \left\{\left[\begin{matrix} B&0\\ 0&A \end{matrix}\right]:B\in\sym(\kg)^K, \; A\in\sym(\pg)^K\right\}.
$$
A simple counterexample of a standard Einstein $G/K$ with $G$ simple to the above property is $\SO(2kl)/K^l$, where $K$ is any simple Lie group and the inclusion is defined as $l$ copies of $\Delta(K)\subset K\times K$ (see Table \ref{tableIA}, 5).  
\end{remark}

\begin{proof}
If $\tau$ also denotes the corresponding $G$-representation, i.e., $\tau(x)A=\Ad(x)A\Ad(x)^{-1}$ for $x\in G$, then every $\tau(x)$ commutes with $\cas_\tau$.  Indeed, $\cas_\tau$ is induced by the Casimir element, which belongs to the center of the enveloping algebra.  Since $\sym(\ggo)^K$ is precisely the intersection of all the kernels of $\tau(z)$, $z\in K$, one obtains that $\cas_\tau\sym(\ggo)^K\subset\sym(\ggo)^K$.  

Concerning the formulas, we observe that $\tau(Z_j)\left[\begin{smallmatrix} B&0\\ 0&A \end{smallmatrix}\right]=0$ 
for all $j$ by the $K$-invariance of $B$ and $A$, and from \eqref{adXnr} one easily obtains that  
$$
\tau(X_i)\left[\begin{matrix} B&0\\ 0&A \end{matrix}\right] = \left[\begin{matrix} 0&-Ba(X_i)+a(X_i)A\\ -a(X_i)^tB+Aa(X_i)^t& [\ad_\pg(X_i),A] \end{matrix}\right].
$$
This implies that
$$
\tau(X_i)^2\left[\begin{matrix} B&0\\ 0&0 \end{matrix}\right] = \left[\begin{matrix} -a(X_i)a(X_i)^tB-Ba(X_i)a(X_i)^t & Ba(X_i)\ad_\pg{X_i}\\  \ast & -2 a(X_i)^tBa(X_i) \end{matrix}\right],
$$
and $\tau(X_i)^2\left[\begin{matrix} 0&0\\ 0&A \end{matrix}\right]$ equals
$$
\left[\begin{matrix} 2a(X_i)Aa(X_i)^t & 
a(X_i)[\ad_\pg(X_i),A]-a(X_i)A\ad_\pg(X_i)\\  
\ast & -a(X_i)^ta(X_i)A - Aa(X_i)^ta(X_i) +[\ad_\pg(X_i), [\ad_\pg(X_i),A]] \end{matrix}\right].
$$
Finally, using Lemma \ref{formulas}, (i), (ii) and (iv), the formulas stated in the lemma follow.  
\end{proof}

\begin{table}
{
$$
\begin{array}{c|c|c|c|c}
\ggo & \dim{\ggo} & \lambda_\tau & \lambda_\tau^{\midop} & \lambda_\tau^{\max}   
\\[2mm] \hline \hline \rule{0pt}{14pt}
\sug(2) & 3 & 3 & - & - 
\\[2mm]  \hline \rule{0pt}{14pt}
\sug(n), n\geq 3 & n^2-1 & 1 & \tfrac{2(n-1)}{n} & \tfrac{2(n+1)}{n} 
\\ [2mm] \hline\rule{0pt}{14pt}
\sog(7) & 21 & \tfrac{6}{5}  & \tfrac{7}{5} & \tfrac{12}{5} 
\\[2mm] \hline\rule{0pt}{14pt}
\sog(n), n\geq 8 & \tfrac{n(n-1)}{2} & \tfrac{n}{n-2}  & \tfrac{2(n-4)}{n-2} & \tfrac{2(n-1)}{n-2}   
\\[2mm] \hline\rule{0pt}{14pt}
%
\spg(n), n\geq 2 & n(2n+1) & \tfrac{n}{n+1} & \tfrac{2n+1}{n+1} & \tfrac{2n+4}{n+1} 
\\[2mm] \hline\rule{0pt}{14pt}
%
\eg_6 & 78 & \tfrac{3}{2} & - & \tfrac{13}{6}  
\\[2mm] \hline\rule{0pt}{14pt}
 \eg_7 & 133 & \tfrac{14}{9} & - & \tfrac{19}{9}  
\\[2mm] \hline\rule{0pt}{14pt}
\eg_8 & 248 & \tfrac{8}{5} & - & \tfrac{31}{15} 
\\[2mm] \hline\rule{0pt}{14pt}
\fg_4 & 52 & \tfrac{13}{9} & - & \tfrac{20}{9} 
\\[2mm] \hline\rule{0pt}{14pt}
\ggo_2 & 14 & \tfrac{7}{6} & - & \tfrac{5}{2} 
\\[2mm] \hline\hline
\end{array}
$$}
\caption{Eigenvalues of the Casimir operator $\cas_\tau$ acting on $\sym_0(\ggo)$ for a compact simple Lie algebra $\ggo$.} \label{table1}
\end{table}

The following are stability criteria in which the Einstein constant is not involved. 

\begin{theorem}\label{sc2}
Suppose that $g_{\kil}$ is Einstein.  
\begin{enumerate}[{\rm (i)}]
\item If $\dim{\kg} < \tfrac{\dim{\ggo}(\lambda_\tau-1)}{2\lambda_\tau^{\max}}$, then $2\rho<\lambda_\pg$ and $g_{\kil}$ is $G$-stable.  

\item If $\dim{\kg} = \tfrac{\dim{\ggo}(\lambda_\tau-1)}{2\lambda_\tau^{\max}}$, then $2\rho\leq\lambda_\pg$ and $g_{\kil}$ is $G$-semistable.  

\item If $\dim{\kg} > \tfrac{\dim{\ggo}(\lambda_\tau^{\max}-1)}{2\lambda_\tau}$, then $\lambda_\pg^{\max} < 2\rho$ and $g_{\kil}$ is $G$-unstable and a local minimum of $\scalar|_{\mca_1^G}$. 

\item If $\dim{\kg} = \tfrac{\dim{\ggo}(\lambda_\tau^{\max}-1)}{2\lambda_\tau}$, then $\lambda_\pg^{\max}\leq 2\rho$.  In particular, $g_{\kil}$ is $G$-unstable if $\lambda_\pg<\lambda_\pg^{\max}$, and it is either $G$-unstable and a local minimum or $G$-neutrally stable if $\lambda_\pg=\lambda_\pg^{\max}$. 
\end{enumerate}
\end{theorem}

\begin{proof}
Recall that $\la A,B\ra=\tr{AB}$ for all $A,B\in\sym(\pg)^K$.  If $\ricci(g_{\kil}) = \rho g_{\kil}$, or equivalently, $\cas_\chi=(2\rho-\unm)I_\pg$, then it follows from Lemma \ref{castau} that
$$
\la\lic_\pg A,A\ra = \unm\la\cas_{\tau}\overline{A},\overline{A}\ra - (2\rho - \unm)|A|^2,  \qquad\forall A\in\sym(\pg)^K, \quad \overline{A}:=\left[\begin{matrix} 0&0\\ 0&A \end{matrix}\right]. 
$$
Since
$$
\lambda_\tau\leq \la\cas_{\tau}\overline{A},\overline{A}\ra\leq \lambda_\tau^{\max}, \qquad\forall A\in\sym_0(\pg)^K, \quad |A|=1,
$$ 
we obtain the following estimates: 
\begin{equation}\label{bounds}
\unm(\lambda_\tau-(8\rho-1))+2\rho \leq \lambda_\pg \leq \lambda_\pg^{\max} \leq  \unm(\lambda_\tau^{\max}-(8\rho-1))+2\rho.   
\end{equation} 
This already proves the criteria given in Corollary \ref{sc1} below.  To conclude the proof, we need to estimate $\lambda_\pg$ and $\lambda_{\pg}^{\max}$ in terms of $\rho$.  

We consider the orthogonal decomposition 
\begin{equation}\label{symg}
\RR I_0 \oplus \overline{\sym_0(\kg)^K} \oplus \overline{\sym_0(\pg)^K} \subset \sym_0(\ggo)^K,
\end{equation}
where
$$
I_0:=\left[\begin{matrix} I_\kg &0\\ 0& -\tfrac{\dim{\kg}}{d}I_\pg \end{matrix}\right], \quad \overline{\sym_0(\kg)^K}:= \left[\begin{matrix} \sym_0(\kg)^K&0\\ 0&0 \end{matrix}\right], \quad \overline{\sym_0(\pg)^K} := \left[\begin{matrix} 0&0\\ 0&\sym_0(\pg)^K \end{matrix}\right].  
$$
Again by Lemma \ref{castau},
\begin{equation}\label{CI02}
\cas_\tau I_0 = \left[\begin{matrix} -\tfrac{2\dim{\ggo}}{d} \kil_\chi &0\\ 0& -\tfrac{2\dim{\ggo}}{d}(2\rho-\unm)I_\pg \end{matrix}\right],
\end{equation}
and so by \eqref{ci}, 
\begin{align*}
\la\cas_\tau I_0,I_0\ra &= -\tfrac{2\dim{\ggo}}{d} \tr{\kil_\chi} + \tfrac{2\dim{\ggo}\dim{\kg}}{d}(2\rho-\unm) 
= 2\dim{\ggo}(2\rho-\unm) + \tfrac{2\dim{\ggo}\dim{\kg}}{d}(2\rho-\unm) \\ 
&= 2\dim{\ggo}(1+\tfrac{\dim{\kg}}{d})(2\rho-\unm) = \tfrac{2(\dim{\ggo})^2}{d}(2\rho-\unm).
\end{align*}
Since
$$
|I_0|^2 = \dim{\kg} + \tfrac{(\dim{\kg})^2}{d} = \tfrac{\dim{\kg}\dim{\ggo}}{d},
$$
we obtain that
\begin{equation}\label{CI0} 
\tfrac{\la\cas_\tau I_0,I_0\ra}{|I_0|^2} = \tfrac{\dim{\ggo}}{\dim{\kg}}(4\rho-1).  
\end{equation}
Thus
$$ 
\lambda_\tau\leq \tfrac{\dim{\ggo}}{\dim{\kg}}(4\rho-1) \leq \lambda_\tau^{\max}, 
$$
or equivalently,
$$ 
\tfrac{2\dim{\kg}}{\dim{\ggo}}\lambda_\tau + 1 \leq 8\rho-1 \leq \tfrac{2\dim{\kg}}{\dim{\ggo}}\lambda_\tau^{\max} +1.  
$$
The condition in part (i) therefore implies that $8\rho-1<\lambda_\tau$, and so by \eqref{bounds} we conclude that $2\rho < \lambda_\pg$, i.e., $g_{\kil}$ is $G$-stable.  The other parts follow in much the same way.  
\end{proof}

\begin{example}
As a first example of the usefulness of the above criteria we consider the full flag manifold $E_8/T^8$.  It follows from Table \ref{table1} that 
$$
\dim{\kg}=8 < 36 = \tfrac{\dim{\ggo}(\lambda_\tau-1)}{2\lambda_\tau^{\max}},
$$
and so according to Theorem \ref{sc2}, (i), the standard metric $g_{\kil}$ on $E_8/T^8$ is $G$-stable and in particular a local maximum of $\scalar|_{\mca_1^G}$.  This is remarkable, considering that $\dim{\mca_1^G}=119$.  On the other hand, consider the space $E_8/\SU(2)^8$.  The Einstein constant $\rho$ is not available in the literature and its computation is far from straightforward (see \S\ref{B.15}), but since
$\dim{\kg}=24 < 36$, we also obtain from Theorem \ref{sc2}, (i) that $g_{\kil}$ is $G$-stable and a local maximum.  Note that $\dim{\mca_1^G}=13$ in this case.  
\end{example}

The numbers on the right in each of the conditions in Theorem \ref{sc2} only depend on $\ggo$ and are given in Table \ref{table-sc} for some Lie algebras (see Table \ref{table1} and Figure \ref{table-lambdas}).  A simple comparison between these numbers and $\dim{\kg}$ gives the following corollary of Theorem \ref{sc2}, which contains the totality of the cases where the criteria can be successfully applied.  

\begin{corollary}\label{sc2-cor}
The standard metric is 
\begin{enumerate}[{\small $\bullet$}] 
\item $G$-stable for the spaces numbered as 8, 12, 13, 15, 16, 17, 18a, 18b and 18c in Tables \ref{tableIB2} and \ref{tableIB3}, 

\item $G$-unstable and a local minimum for case 1 in Table \ref{tableIB1}, 

\item $G$-semistable for case 2 in Table \ref{tableIB1} and case 10 in Table \ref{tableIB2}.   
\end{enumerate}
\end{corollary}

The following stability criteria involving the Einstein constant follows from \eqref{bounds}, recall that $1\leq 8\rho-1\leq 3$ (see Figure \ref{table-lambdas}).   

\begin{corollary}\label{sc1}
Assume that $\ricci(g_{\kil}) = \rho g_{\kil}$.
\begin{enumerate}[{\rm (i)}] 
\item If $8\rho-1<\lambda_\tau$, then $2\rho<\lambda_\pg$ and $g_{\kil}$ is $G$-stable.  

\item If $8\rho-1=\lambda_\tau$, then $2\rho\leq\lambda_\pg$ and $g_{\kil}$ is $G$-semistable.  

\item If $\lambda_\tau^{\max}<8\rho-1$, then $\lambda_\pg^{\max}<2\rho$ and $g_{\kil}$ is $G$-unstable and a local minimum.  

\item If $\lambda_\tau^{\max}=8\rho-1$, then $\lambda_\pg^{\max}\leq 2\rho$ (cf.\ Theorem \ref{sc2}, (iv)).   
\end{enumerate}
\end{corollary}

We note that each of the parts of the above corollary is a necessary condition for the corresponding part in Theorem \ref{sc2} to hold.  The following example shows that the estimates \eqref{bounds} are not in general sharp.  This warns us that all the above criteria may fail to find the $G$-stability type in many cases.  

\begin{table}
{
$$
\begin{array}{c|c|c|c|c|c}
& \sog(8) & \fg_4 & \eg_6 & \eg_7 & \eg_8 
\\[2mm] \hline \hline \rule{0pt}{14pt}
\tfrac{\dim{\ggo}(\lambda_\tau-1)}{2\lambda_\tau^{\max}} &\tfrac{26}{5} & 14 & 9 & \tfrac{35}{2} & 36 
\\[2mm] \hline \rule{0pt}{14pt}
\tfrac{\dim{\ggo}(\lambda_\tau^{\max}-1)}{2\lambda_\tau} & 22 & 14 & \tfrac{91}{3} & \tfrac{95}{2} & \tfrac{248}{3} 
%
\\[2mm] \hline\hline
\end{array}
$$}
\caption{Relevant quantities to apply Theorem \ref{sc2}.} \label{table-sc}
\end{table}

\begin{figure}
\begin{tikzpicture}[scale=1.2]

\draw[-,thick] (0,0) -- (10,0);
\draw[thick, fill] (0,0) node[below] {{\tiny $0$}};
\draw[-,thick] (0,0) node{{\tiny $|$}};
\draw[thick, fill] (3,0) node[below] {{\tiny $1$}};
\draw[-,thick] (3,0) node{{\tiny $|$}};
\draw[thick, fill] (6,0) node[below] {{\tiny $2$}};
\draw[-,thick] (6,0) node{{\tiny $|$}};
\draw[thick, fill] (9,0) node[below] {{\tiny $3$}};
\draw[-,thick] (9,0) node{{\tiny $|$}};

\draw[thick, fill, red] (2,0) circle (1pt) node[below] {{\tiny $\tfrac{2}{3}$}};
\draw[thick, fill, red] (3,0) circle (1pt) node[below] {{\tiny $1$}};
\draw[thick, fill, red] (4,0) circle (1pt) node[below] {{\tiny $\tfrac{4}{3}$}};

\draw[thick, fill, red] (4.8,0) circle (1pt); 
\draw (4.8,0) node[below] {{\tiny $\eg_8$}};
\draw[thick, fill, red] (4.5,0) circle (1pt); 
\draw (4.5,0) node[below] {{\tiny $\eg_6$}};
\draw[thick, fill, red] (4.66,0) circle (1pt); 
\draw (4.66,0) node[above] {{\tiny $\eg_7$}};
\draw[thick, fill, red] (4.33,0) circle (1pt); 
\draw (4.33,0) node[above] {{\tiny $\fg_4$}};

\draw[thick, fill, blue] (7,0) circle (1pt) node[below] {{\tiny $\tfrac{7}{3}$}};
\draw[thick, fill, blue] (8,0) circle (1pt) node[below] {{\tiny $\tfrac{8}{3}$}};

\draw[thick, fill, blue] (6.5,0) circle (1pt); 
\draw (6.5,0) node[above] {{\tiny $\eg_6$}};
\draw[thick, fill, blue] (6.33,0) circle (1pt); 
\draw (6.33,0) node[below] {{\tiny $\eg_7$}};
\draw[thick, fill, blue] (6.2,0) circle (1pt); 
\draw (6.2,0) node[above] {{\tiny $\eg_8$}};
\draw[thick, fill, blue] (6.66,0) circle (1pt); 
\draw (6.66,0) node[below] {{\tiny $\fg_4$}};

\draw[-,thick, olive] (3,0.5) -- (9,0.5); 
\draw[-,thick, olive] (3,0.5) node{{\tiny $($}};
\draw[-,thick, olive] (9,0.5) node{{\tiny $)$}};
\draw (6,0.5) node[above] {{\tiny $8\rho-1$}}; 
\draw (3,0) node[above] {{\tiny $\sug(n)$}};
\draw[-,thick, red] (2,-0.6) -- (3,-0.6); 
\draw[-,thick, red] (3,-0.6) node{{\tiny $)$}};
\draw[-,thick, red] (2,-0.6) node{{\tiny $[$}};
\draw (2.5,-0.6) node[below] {{\tiny $\spg(n)$}};
\draw[-,thick, red] (3,-1) -- (4,-1); 
\draw[-,thick, red] (3,-1) node{{\tiny $($}};
\draw[-,thick, red] (4,-1) node{{\tiny $]$}};
\draw (3.5,-1) node[below] {{\tiny $\sog(n)$}};
\draw[-,thick, blue] (6,-0.6) -- (7,-0.6); 
\draw[-,thick, blue] (6,-0.6) node{{\tiny $($}};
\draw[-,thick, blue] (7,-0.6) node{{\tiny $]$}};
\draw (6.5,-0.6) node[below] {{\tiny $\sog(n)$}};
\draw[-,thick, blue] (6,-1) -- (8,-1); 
\draw[-,thick, blue] (6,-1) node{{\tiny $($}};
\draw[-,thick, blue] (8,-1) node{{\tiny $]$}};
\draw (7,-1) node[below] {{\tiny $\sug(n)$, $\spg(n)$}};

\draw[red] (3.5,-0.6) node {{\tiny $\lambda_\tau$}};
\draw[blue] (7.5,-0.6) node {{\tiny $\lambda_\tau^{\max}$}};
\end{tikzpicture}
\caption{Stability criteria.}\label{table-lambdas}
\end{figure}
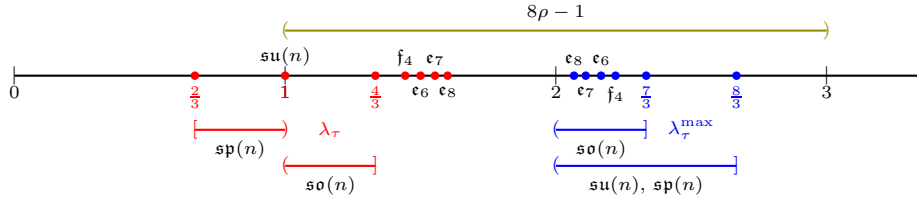

\begin{example}\label{sharp}
For the full flag manifold $M=\SU(n)/T^{n-1}$, $n\geq 3$, one has that $\rho=\tfrac{n+2}{4n}$, $\lambda_\pg=\unm$ and $\lambda_\pg^{\max} = \tfrac{n-1}{n}$ (see \cite[Section 6]{stab-tres}).  From Table \ref{table1} we know that $\lambda_\tau=1$ and $\lambda_\tau^{\max}=\tfrac{2(n+1)}{n}$, which implies that
$$
\unm\lambda_\tau-2\rho+\unm = \tfrac{n-2}{2n} < \unm = \lambda_\pg, \qquad 
\lambda_\pg^{\max} = \tfrac{n-1}{n} < 1 = \unm\lambda_\tau^{\max}-2\rho+\unm.
$$
It is also easy to see that $\lambda_\pg<\lambda_\tau^{\midop}$ if and only if $n\geq 9$ and that $\lambda_\pg=\lambda_\tau^{\midop}$ if and only if $n=8$.  Furthermore, since $\lambda_\tau^{\max}>\tfrac{n+4}{n}=8\rho -1$, the $G$-instability of $g_{\kil}$ (i.e., $\lambda_\pg<2\rho$) does not follow from Corollary \ref{sc1}, (iii) or (iv).   
\end{example}

The following are the only cases, beyond those covered by Corollary \ref{sc2-cor}, where we can apply Corollary \ref{sc1} after computing the Einstein constant $\rho$:

\begin{enumerate}[(i)] 
\item $\spg(3n-1)/\spg(n)\oplus\ug(2n-1)$ (see Table \ref{tableIA}, 7a): 
we have that $d_1= 2n(2n-1)$, 
$d_2 = 4n(2n-1)$ (see \cite[III.6]{DckKrr}), and it follows from \cite[Table 1, pp.37]{DtrZll} that $\kil_{\spg(n)}=\tfrac{n+1}{3n} \kil_{\spg(3n+2)}|_{\spg(n)}$.  On the other hand, since $\kil_{\sug(n)} = \tfrac{4n}{8(n+1)} \kil_{\spg(n)}$, we obtain that  
$$
\kil_{\sug(2n-1)} = \tfrac{4(2n-1)}{16n} \kil_{\spg(2n-1)}
= \tfrac{2n-1}{4n} \tfrac{2}{3} \kil_{\spg(3n-1)}
= \tfrac{2n-1}{6n} \kil_{\spg(3n-1)}.  
$$
Formula \eqref{rhoci} therefore gives that $\rho=\tfrac{5}{12}$.  This implies that $8\rho-1=\tfrac{7}{3}\geq\tfrac{2(3n+1)}{3n}=\lambda_\tau^{\max}$ if and only if $n\geq 2$, where equality holds if and only if $n=2$.  Thus $g_{\kil}$ is $G$-unstable and a local minimum for any $n\geq 3$ (see \S\ref{A.7a} for $n=1,2$).

\item $\sog(3n+2)/\sog(n)\oplus\ug(n+1)$ (see Table \ref{tableIA}, 7b): 
since $d_1= n(n+1)$ and $d_2 = 2n(n+1)$ (see \cite[I.18]{DckKrr}), one obtains from \cite[Table 1, pp.37]{DtrZll} that 
$$
\kil_{\sog(n)} = \tfrac{n-2}{3n} \kil_{\sog(3n+2)}|_{\sog(n)}, \qquad \kil_{\sug(n+1)}  = \tfrac{n+1}{3n}\kil_{\sog(3n+2)}|_{\sug(n+1)}.
$$  
Thus $\rho=\tfrac{5}{12}$ by formula \eqref{rhoci} and so $8\rho-1=\tfrac{7}{3}>\tfrac{2(3n+1)}{3n}=\lambda_\tau^{\max}$ for any $n\geq 3$, which implies that $g_{\kil}$ is $G$-unstable and a local minimum.

\item $\eg_6/\sping(8)\oplus\RR^2$ (see Table \ref{tableIB1}, 3): 
it is known that $\rho=\tfrac{5}{12}$, so $8\rho-1=\tfrac{7}{3}>\tfrac{13}{6}=\lambda_\tau^{\max}$ and hence $g_{\kil}$ is $G$-unstable and a local minimum (cf.\ \cite[Table 2, W7]{stab-dos}).
\end{enumerate}

The information on the applicability of the criteria given in Corollaries \ref{sc1} and \ref{sc2-cor} was respectively added on the columns C1 and C2 in Tables \ref{tableIB1}-\ref{tableIB3} and \ref{tableIAA}.  Note that C1 necessarily works if C2 does it.  

\begin{remark}
The authors recently became aware of the following stability criterion given in \cite[Theorem 2]{Nkn}: if $\rho>\frac{2}{5}$ then $g_{\kil}$ is a local minimum of $\scalar|_{\mca_1^G}$.  It is easy to check that this works for $6$ of the $9$ local minima exhibited in Tables \ref{tableIA}-\ref{tableIB3}, the three exceptions are the cases Table \ref{tableIA}, 6,8 and Table \ref{tableIB1}, 6.  We note that $\rho>\frac{2}{5}$ implies that $\frac{11}{5}<8\rho-1$, which in turn gives that $\lambda_\tau^{\max}<8\rho-1$ for $\eg_6$, $\eg_7$, $\eg_8$ and some low dimensional classical simple Lie algebras, establishing the local minimality of the standard metric by Corollary \ref{sc1}, (iii) in the case when $\ggo$ is one of these Lie algebras.  
\end{remark}

\section{Full flag manifolds $\SO(2n)/T^n$}\label{so2n-sec}

We study in this section the homogeneous spaces $\SO(2n)/T^n$, $n\geq 3$ (see (Table \ref{tableIA}, 1b.1 and 1b.2), where $T^n$ is the usual maximal torus of $\SO(2n)$, by following the lines of \cite[Section 6]{stab-tres}.  The standard block matrix reductive decomposition is given by 
$$
\ggo=\kg\oplus\pg, \qquad\pg=\pg_{12}\oplus\pg_{13}\oplus\dots\oplus\pg_{(n-1)n},
$$
where every block $\pg_{ij}=\pg_{ji}$ (note that always $i\ne j$) has dimension $4$ and is $\Ad(T^n)$-invariant.   It is easy to see that each of these subspaces in addition decomposes in the sum of two $\Ad(T^n)$-irreducible and pairwise inequivalent $2$-dimensional subspaces.  Thus $\SO(2n)/T^n$ is multiplicity-free and $\dim{\mca^G}=n(n-1)$.  

It is easy to check that $[\pg_{ij},\pg_{kl}]_\pg=0$ if $\{ i,j\}$ and $\{ k,l\}$ are either equal or disjoint, and that there exist decompositions $\pg_{ij} = \pg_{ij}^1\oplus\pg_{ij}^2$ such that 
\begin{align*}
[\pg_{ij}^1,\pg_{ik}^1]_\pg&\subset\pg_{jk}^1, &
[\pg_{ij}^1,\pg_{ik}^2]_\pg&\subset\pg_{jk}^2, &
[\pg_{ij}^2,\pg_{ik}^2]_\pg&\subset\pg_{jk}^1,
\end{align*} 
for all $j\ne k$.  Moreover, it is also easy to prove that all the above triples produces the same nonzero structural constant $\tfrac{1}{2(n-1)}$.   It follows from \eqref{Lpijk} that the Lichnerowicz Laplacian restricted to $\sca^2(M)^G$ of the standard metric $g_{\kil}$, which is Einstein with $\rho=\tfrac{n}{4(n-1)}$, is given by 
$$
[\lic_\pg]_{(ij)^1(ij)^1} = [\lic_\pg]_{(ij)^2(ij)^2} = \tfrac{n-2}{n-1}, \quad [\lic_\pg]_{(ij)^1(ik)^1} =[\lic_\pg]_{(ij)^1(ik)^2}= [\lic_\pg]_{(ij)^2(ik)^2} = -\tfrac{1}{4(n-1)}, 
$$
for all $\#\{ i,j,k\}=3$, and zero otherwise.  This implies that 
\begin{equation}\label{Lp-so2n}
[\lic_\pg] = \tfrac{1}{4(n-1)}\Big(4(n-2)I-
\left[\begin{matrix} \Adj(X)&\Adj(X)\\ \Adj(X)&\Adj(X)\end{matrix}\right]\Big), 
\end{equation}  
where $X=J(n,2,1)$ is the Johnson graph with parameters $(n,2,1)$ (see \cite[Section 1.6]{GdsRyl}) and $\Adj(X)$ denotes its adjacency matrix.  

Since this graph is strongly regular with parameters $(\tfrac{n(n-1)}{2},2(n-2),n-2,4)$ for any $n\geq 4$ (see \cite[Section 10.1]{GdsRyl}), it follows from \cite[Section 10.2]{GdsRyl} that the spectrum of $\Adj(X)$ is given by 
$$
2(n-2), \quad n-4, \quad -2,\quad \mbox{with multiplicities}\quad  1, \quad n-1, \quad \tfrac{n(n-3)}{2},  
$$    
respectively, and thus the matrix $\left[\begin{smallmatrix} \Adj(X)&\Adj(X)\\ \Adj(X)&\Adj(X)\end{smallmatrix}\right]$ has eigenvalues   
$$
4(n-2), \quad 2(n-4), \quad -4, \quad 0, \quad \mbox{with multiplicities}\quad  1, \quad n-1, \quad \tfrac{n(n-3)}{2}, \quad \tfrac{n(n-1)}{2}.  
$$    
respectively.  It follows from \eqref{Lp-so2n} that
$$
\Spec(\lic_\pg)=\left\{ 0,  \tfrac{n}{2(n-1)}, 1, \tfrac{n-2}{n-1}\right\}, 
$$
with multiplicities $1$, $n-1$, $\tfrac{n(n-3)}{2}$ and $\tfrac{n(n-1)}{2}$, respectively.  Thus $\lambda_\pg = \tfrac{n}{2(n-1)}$, $\lambda_\pg^{\midop} = \tfrac{n-2}{n-1}$ and $\lambda_\pg^{\max} =1$ for any $n\geq 4$.  

For $n=3$, $X$ is the complete graph on $3$ vertices and so the spectrum of $\Adj(X)$ equals $\{ 2, -1\}$, with multiplicities $1$ and $2$, respectively.  Thus $\lambda_\pg=\unm$ and $\lambda_\pg^{\max}=\tfrac{3}{4}$ with multiplicities $3$ and $2$, respectively.  

Since $2\rho=\tfrac{n}{2(n-1)}$, we conclude that $g_{\kil}$ is $G$-neutrally stable of nullity $n-1$ for any $n\geq 4$ and it is $G$-unstable of coindex $3$ and $G$-degenerate of nullity $2$ for $n=3$.

\section{Two isotropy summands}\label{r2-sec}

We study in this section standard Einstein metrics on homogeneous spaces with only two irreducible isotropy summands, i.e., $\pg=\pg_1\oplus\pg_2$.  These spaces were classified in \cite{DckKrr}, and some corrections were added in \cite[Appendix A]{He}.  

\begin{remark}
The authors take the opportunity to note that the space $E_8/\Spin(9)$ (see Table \ref{tableIB2}, 10) was missed in the above two papers.  
\end{remark}

We assume that the summands are inequivalent since, according to \cite{DckKrr}, the only case where this is not the case is $\sog(8)/\ggo_2$ (see Table \ref{tableIA}, 9), which will be treated at the end of the section.  Since $\dim{\mca_1^G}=1$, we have that $\scalar|_{\mca_1^G}$ is a one-variable function and so $\Spec(\lic_\pg)=\{ 0,\lambda_\pg\}$.  It follows from \cite[Section 3.1]{stab-dos} that for the standard metric $g_{\kil}$ one has 
\begin{equation}\label{r2-rho}
2\rho =1 -\tfrac{[111]+[122]+2[112]}{2d_1} = 1 -\tfrac{[222]+[112]+2[122]}{2d_2}, 
\end{equation}
and
\begin{equation}\label{r2-lambda}
\lambda_\pg = \tfrac{d_1+d_2}{d_1d_2} \left([112]+[122]\right).   
\end{equation}

If $K$ is not a maximal subgroup of $G$, then it can be assumed that $[112]=0$ (i.e., $\kg\oplus\pg_1$ is a subalgebra) and it is known that there exist at most two $G$-invariant Einstein metrics on $G/K$ (see \cite[Theorem (3.1)]{WngZll} or \cite[Theorem 3.1]{DckKrr}).  

On the other hand, if $K$ is a maximal subgroup of $G$, then $[112],[122]>0$, there exist at least one (a global maximum, see \cite[Theorem (2.2)]{WngZll}) and at most three $G$-invariant Einstein metrics on $G/K$ (see \cite[Section 3.2]{DckKrr}) and $g_{\kil}$ is Einstein in exactly the following four cases (see \cite[Section 6]{DckKrr}):  Table \ref{tableIA}, items 3a and 3b and Table \ref{tableIB2}, items 8 and 11.  However, we do not know a priori whether $g_{\kil}$ is a global maximum or not.  

\begin{lemma}\label{r2-Kmax}
Assume that $g_{\kil}$ is Einstein, $d_1=d_2$, $[111]=[222]$ and $[112]=[122]>0$.  Then, the following conditions are equivalent:  
\begin{enumerate}[{\rm (i)}] 
\item $g_{\kil}$ is a global maximum. 

\item $g_{\kil}$ is the unique $G$-invariant Einstein metric.  
\end{enumerate}   
Otherwise, $g_{\kil}$ is a local minimum and there exist other two Einstein metrics (both necessarily local maxima and at least one of them global maximum).    
\end{lemma}

\begin{proof}
If $a:=\frac{d_1}{2}-\unc[111]-\unm[112]$ and $b:=\unc[112]$ (note that $a>b>0$), then the scalar curvature $\scalar(x)$ of the metric $(x,\frac{1}{x})$ satisfies (see e.g. \cite[(19)]{stab-dos})
\begin{align*}
\scalar(x) &=  a\left(x+\tfrac{1}{x}\right) - b\left(x^3+\tfrac{1}{x^3}\right), \\ 
\scalar'(x) &=  a\left(1-\tfrac{1}{x^2}\right) - 3b\left(x^2-\tfrac{1}{x^4}\right), \\ 
\scalar''(x) &=  2a\tfrac{1}{x^3} - 6bx -12b\tfrac{1}{x^5}.   
\end{align*}
Thus $\scalar(x)\to-\infty$ as $x\to 0$ or $x\to\infty$.  Since $\scalar''(1)=2a-18b$, $g_{\kil}$ is $G$-stable if and only if $a<9b$ (note that this is equivalent to $2\rho<\lambda_\pg$ and to $d_1\rho<2[112]$).  

It also follows that the squares of the other potential critical points of $\scalar$ are zeroes of $z^2+\frac{3b-a}{3b}z+1$, so either $a\leq 9b$ and $g_{\kil}$ is the unique critical point (necessarily a global maximum, and $G$-degenerate if and only if $a=9b$) or $a>9b$ and there exist other two Einstein metrics $x_1<1<x_2$ (both necessarily local maxima and at least one of them global maximum) and $g_{\kil}$ is a local minimum, concluding the proof.  
\end{proof}

In what follows, we find the $G$-stability type of the standard metric in many cases.

\subsection{$\sog(n^2)/\sog(n)\oplus\sog(n)$}\label{A.3a} 
(See Table \ref{tableIA}, 3a).  
The inclusion $K\subset G$ is defined via the representation $\RR^n\otimes\RR^n$ of $\SO(n)\times\SO(n)$, so $d_1=d_2=\frac{n(n-1)^2(n+2)}{4}$.  According to \cite[Section 2]{WngZll2}, the Casimir constant is given by $E(\chi_i)/\alpha_G$ using the notation in that paper, so in this case, it follows from \cite[Tables in pp.602 and pp.583]{WngZll2} that it is $\frac{2(n-1)}{n(n^2-2)}$.   The Einstein constant is therefore given by $\rho=\unc+\frac{n-1}{n(n^2-2)}$.  

The automorphism $(X,Y)\mapsto (Y,X)$ of $\so(n)\oplus\so(n)$ forces $[111]=[222]$ and $[112]=[122]$. 
Furthermore, \eqref{r2-rho} gives an expression for $[111]+3[112]$, thus it is sufficient (and necessary) to determine any of these two nonzero structural constants to know the whole picture. 

\begin{lemma}
We have that $[112]=\tfrac{n(n-1)^2(n-2)(n+2)^2}{16(n^2-2)}$.
\end{lemma}

\begin{proof}
The proof involves tedious but straightforward calculations. 
We next explain the strategy leaving the details to the reader. 

We use the identification $\R^{n^2}=\R^n\otimes\R^n$ with basis $e_i\otimes e_j$ for $1\leq i,j\leq n$. 
Since $\kil_{\so(n^2)}(X,Y)=(n^2-2)\tr(XY)$ for all $X,Y\in\so(n^2)$, $X_{(ij),(kl)}:=\tfrac{1}{\sqrt{2(n^2-2)}} (E_{(ij),(kl)}-E_{(kl),(ij)})$ for $i<k$ and $j<l$ form a basis of $\so(n^2)$. 
Here, $E_{(ij),(kl)}$ is the $n^2\times n^2$ matrix with a $1$ in the entry $((ij),(kl))$ and zero otherwise.  Let $h$ be in $\{1,2\}$.
It is not difficult to see that an orthonormal basis for $\fp_h$ is given by $\mathcal B^{(h)}:=\mathcal B_1^{(h)}\cup \mathcal B_2^{(h)}$, where 
$
\mathcal B_1^{(h)} = \{
\tfrac{1}{\sqrt{2}} (X_{(ij),(kl)} -(-1)^h X_{(il),(kj)}): i<k,\, j<l\}
$, 
$
\mathcal B_2^{(1)} =\{
\sum_{t=1}^{n-1} (\delta_{t,j}-\tfrac{1}{n+\sqrt{n}}) X_{(it),(kt)} -\tfrac{1}{\sqrt{n}} X_{(in),(kn)} : i<k,\, j<n\}
$, and 
$
\mathcal B_2^{(2)} =\{
\sum_{t=1}^{n-1} (\delta_{t,j}-\tfrac{1}{n+\sqrt{n}}) X_{(ti),(tk)} -\tfrac{1}{\sqrt{n}} X_{(ni),(nk)} : i<k,\, j<n\}
$.
It follows that
\begin{equation*} 
[12h]
= 
\sum_{1\leq p,q\leq 2}
\sum_{X_1\in \mathcal B_p^{(1)}} 
\sum_{X_2\in \mathcal B_q^{(2)}}
\sum_{X_3\in\mathcal B^{(h)}} 
\kil  \big(
	[X_1,X_2],X_3 \big)^2
=: 
\sum_{1\leq p,q\leq 2} \textbf{B}(p,q).
\end{equation*}
Tiresome computations give $\textbf{B}(1,1)=\tfrac{n^3(n-1)^2(n-2)}{16(n^2-2)}$,
$\textbf{B}(1,2)=\textbf{B}(2,1)=\tfrac{n^2(n-1)^2(n-2)}{8(n^2-2)}$, and $\textbf{B}(2,2)=\tfrac{n(n-1)^2(n-2)}{4(n^2-2)}$, and the assertion follows. 
\end{proof}

Since $\lambda_\pg = \tfrac{4}{d_1}[112]=\frac{n^2-4}{n^2-2}$ by \eqref{r2-lambda}, it is easy to see that $\lambda_\fp >2\rho$ if and only if $n^3-10n+4>0$, which is true for every $n\geq3$. 
We conclude that $g_{\kil}$ is $G$-stable.

\begin{remark}
To the best of authors' knowledge, the classification of $G$-invariant Einstein metrics on $G/K$ was unknown.  Lemma~\ref{r2-Kmax} implies that $g_{\kil}$ is the unique $G$-invariant Einstein metric and a global maximum. 
\end{remark}

\subsection{$\sog(4n^2)/\spg(n)\oplus\spg(n)$}\label{A.3b} 
(See Table \ref{tableIA}, 3b).   
The inclusion $K\subset G$ is defined via the real representation of $\Sp(n)\times\Sp(n)$ whose complexification is $\mathbb C^{2n}\otimes\mathbb C^{2n}$, so $d_1=d_2=n(n-1)(2n+1)^2$.  In much the same way as the above case, one obtains from \cite[Tables in pp.602 and pp.583]{WngZll2} that the Casimir constant is $\frac{2n+1}{2n(2n^2-1)}$, so the Einstein constant is given by $\rho=\unc+\frac{2n+1}{4n(2n^2-1)}$.  Again, $[111]=[222]$, $[112]=[122]$, $2\rho = 1-\tfrac{1}{2d_1}([111]+3[112])$ and $\lambda_\pg = \tfrac{4}{d_1}[112]$. 

Although it is sufficient to compute $[111]$ or $[112]$ to determine $\lambda_\pg$ and consequently the $G$-stability type of $(G/K,g_{\kil})$, the authors decided to leave this case open due to the great difficulty they encountered in the computations.

\subsection{$\sug(pq+l)/\sug(p)\oplus\sug(q)\oplus\ug(l)$}\label{A.6} 
(See Table \ref{tableIA}, 6).  
The inclusion $K\subset G$ is defined via the representation $(\CC^p\otimes\CC^q)\oplus\CC^l$ of $(\U(p)\times\U(q))\times\U(l)$ and so $\kg$ is isomorphic to $\sg(\ug(p)\oplus\ug(q)\oplus\ug(l))/\RR Z$ for a certain $Z\in\zg$, the center of $\sg(\ug(p)\oplus\ug(q)\oplus\ug(l))$.  It is proved in \cite[Example 8, pp.576]{WngZll2} that the standard metric is Einstein if and only if $l=\frac{p^2+q^2+1}{pq}$, for which there are infinitely many integer solutions.  Note that $\dim{\kg}= p^2+q^2+l^2-2$ and $d=p^2q^2+p^2+q^2+3$.  We have that $\pg=\pg_1\oplus\pg_2$, where $\pg_2$ is the off diagonal block of $\sug(pq+l)$ and $\pg_1$ is given by $\pg_1=\widetilde{\pg}_1\oplus\RR Z^\perp$, where $\sug(pq)=\sg(\ug(p)\oplus\ug(q))\oplus\widetilde{\pg}_1$ and $\{ Z,Z^\perp\}$ is an orthogonal basis of the center of $\zg$, that is,   
$$
\sug(pq+l)/\zg = 
\left[\begin{matrix}
\sg(\ug(p)\oplus\ug(q))\oplus\widetilde{\pg}_1 & \pg_2 \\ 
\pg_2 & \sug(l)
\end{matrix}\right].  
$$
Thus $d_1= p^2q^2-p^2-q^2+1$, $d_2=2pql=2p^2+2q^2+2$ and the only nonzero structural constants are $[111]$ and $[122]$.  It follows from \eqref{r2-rho} and \eqref{r2-lambda} that 
$$
2\rho = 1-\tfrac{1}{d_2}[122], \qquad \lambda_\pg = \tfrac{d}{d_1d_2}[122].  
$$
On the other hand, since $\kil_{\sug(p)}=\tfrac{1}{q^2}\kil_{\sug(pq)}|_{\sug(p)}$ and $\kil_{\sug(q)}=\tfrac{1}{p^2}\kil_{\sug(pq)}|_{\sug(q)}$, we have that
$$
\kil_{\sug(p)}=\tfrac{p}{q(pq+l)}\kil_{\sug(pq+l)}|_{\sug(p)},  \; \kil_{\sug(q)}=\tfrac{q}{p(pq+l)}\kil_{\sug(pq+l)}|_{\sug(q)},  \; \kil_{\sug(l)}=\tfrac{l}{pq+l}\kil_{\sug(pq+l)}|_{\sug(l)}, 
$$
and hence according to \eqref{rhoci}, $\rho=\unc+\frac{1}{2d}\alpha$, where 
\begin{align*}
\alpha =& \left(1-\tfrac{p}{q(pq+l)}\right)(p^2-1) + \left(1-\tfrac{q}{p(pq+l)}\right)(q^2-1) + \left(1-\tfrac{l}{pq+l}\right)(l^2-1) + 1 \\ 
=& \tfrac{p^4q^2+p^2q^4-2+(p^2+q^2+1)^2+p^2+q^2+1}{pq(pq+l)} 
=\tfrac{(p^2q^2+p^2+q^2+3)(p^2+q^2)}{p^2q^2+p^2+q^2+1},
\end{align*}  
which implies that 
$$
\rho= \frac{p^2q^2+3p^2+3q^2+1}{4(p^2q^2+p^2+q^2+1)}, \qquad \lambda_\pg = \frac{p^2q^2+p^2+q^2+3}{2(p^2q^2+p^2+q^2+1)},
$$
since $[122]=\frac{d_2(d-2\alpha)}{2d}$ and so $\lambda_\pg =\tfrac{d-2\alpha}{2d_1}$.  Now a straightforward calculation gives that $\lambda_\pg<2\rho$ if and only if $1<p^2+q^2$,
and hence $g_{\kil}$ is always $G$-unstable as $p,q\geq 2$.

\subsection{$\spg(3n-1)/\spg(n)\oplus\ug(2n-1)$}\label{A.7a} 
(See Table \ref{tableIA}, 7a).  
This case was already solved for $n\geq 3$ in part (i) at the end of \S\ref{Gstab-sec}.  The intermediate subalgebra $\kg\subset\hg:=\spg(n)\oplus\spg(2n-1)\subset\ggo$ gives the reductive decomposition $\ggo=\kg\oplus\pg_1\oplus\pg_2$, where $\hg=\kg\oplus\pg_1$.  Since $H/K$ and $G/H$ are symmetric spaces, we obtain that the only nonzero structural constant is $[122]$.  It follows from \eqref{r2-rho} that $d_2=2d_1$ and $[122] = (1-2\rho)2d_1$, and from \eqref{r2-lambda} that  
$$
\lambda_\pg = \tfrac{3d_1}{2d_1^2}(1-2\rho)2d_1 = 3(1-2\rho) < 2\rho,
$$
since $2\rho=\tfrac{5}{6}$.  Hence $g_{\kil}$ is $G$-unstable and a local minimum for any $n\geq 1$.

\subsection{$\sog(26)/\spg(1)\oplus\spg(5)\oplus\sog(6)$}\label{A.8} 
(See Table \ref{tableIA}, 8).  
For the intermediate subalgebra $\kg\subset\hg:=\sog(20)\oplus\sog(6)\subset\ggo$, we obtain the reductive decomposition $\ggo=\kg\oplus\pg_1\oplus\pg_2$, where $\hg=\kg\oplus\pg_1$, $H/K$ is isotropy irreducible and $G/H$ is symmetric.  Note that $d_1=132$ and $d_2=120$.  The only nonzero structural constants are therefore $[111]$ and $[122]$.   

In order to apply formula \eqref{rhoci} to compute $\rho$, we set $\kg_1:=\spg(1)$, $\kg_2:=\spg(5)$ and $\kg_3:=\sog(6)$.  We have that $c_3=\frac16$ by \cite[pp.37]{DtrZll} and by following the method given in \cite[pp.38--40]{DtrZll} one obtains that $c_1=\frac{1}{60}$ and $c_2=\unc$.  Thus $\rho=\tfrac{29}{80}$ and it therefore follows from \eqref{r2-rho} that $[122]=33$.  Now \eqref{r2-lambda} gives that 
$$
\lambda_\pg = \tfrac{21}{40} < \tfrac{29}{40} = 2\rho, 
$$
and hence $g_{\kil}$ is $G$-unstable and a local minimum.

\subsection{$\eg_6/\sug(2)\oplus\sog(6)$}\label{B.4} 
(See Table \ref{tableIB1}, 4).  
According to \cite[IV.16]{DckKrr}, $g_{\kil}$ is the only $G$-invariant Einstein metric, hence it is necessarily an inflection point of $\scalar|_{\mca_1^G}$.  In particular, it is $G$-neutrally stable.

\subsection{$\eg_8/\sping(9)$}\label{B.10} (See Table \ref{tableIB2}, 10).  
The embedding of $\fk$ into $\eg_8$ is via the spin representation $\fk\to \so(16)\subset \eg_8$, which implies that the only nonzero structural constants are  $[111]$ and $[122]$.
According to \eqref{rhoci}, $\rho=\tfrac14+\tfrac{36(1-c)}{424}$, where $\kil_{\spin(9)}=c\kil_{\eg_8}|_{\spin(9)}$. 
	Furthermore, one can see that $c=\tfrac{7}{60}$ by using \cite[pp.38--40]{DtrZll}, thus $\rho=\tfrac{13}{40}$.
We have that $d_1=84$ and $d_2=128$, so \eqref{r2-rho} gives that $[122]=\tfrac{224}{5}$.  By \eqref{r2-lambda} we conclude that 
$$
\lambda_\pg = \tfrac{53}{60} > \tfrac{13}{20} = 2\rho, 
$$
and thus $g_{\kil}$ is $G$-stable.

\subsection{$\eg_8/\sug(5)\oplus\sug(5)$}\label{B.11} 
(See Table \ref{tableIB2}, 11).   
It is straightforward to prove that $\kil_{\sug(5)} = \tfrac16 \kil_{\fe_8}|_{\su(5)}$ for both copies, which implies that $\rho=\tfrac{7}{20}$ by using formula \eqref{rhoci}.  It can be shown that the only nonzero structural constants are $[112]$ and $[122]$, and since $d_1=d_2=100$ and $\rho=\tfrac{7}{20}$, one obtains that $[112]=[122]=20$ by \eqref{r2-rho}.  This implies that $\lambda_\pg = \tfrac{4}{5}$.  Since $2\rho<\lambda_\pg$ we obtain that $g_{\kil}$ is $G$-stable; moreover, it is the only Einstein metric on $G/K$ and a global maximum by Lemma \ref{r2-Kmax}.

\subsection{$\sog(8)/\ggo_2$}\label{A.9} 
(See Table \ref{tableIA}, 9).  
Note that $\SO(8)/G_2=S^7\times S^7$ as a manifold.  It follows from \cite[Section 5]{Krr} that $d_1=d_2=7$, $[112]=[222]=0$ and $[111]=[122]=\tfrac{7}{6}$, thus $\rho=\tfrac{5}{12}$ by \eqref{r2-rho}.  We now use \eqref{r2-lambda} and \cite[Remark 3.3]{stab-dos} to conclude that 
$$
\lambda_\pg\leq \tfrac{1}{3} = \tfrac{2}{7}\tfrac{7}{6}  <  \tfrac{5}{6} = 2\rho,
$$
and hence $g_{\kil}$ is $G$-unstable.  It can be proved that $g_{\kil}$ is actually a local minimum (see \cite{Gtr}).

\section{Case $\SO(\nn)/K_1\times\dots\times K_l$}\label{som-sec}

In this section, we show that all the standard Einstein metrics constructed in \cite[Example 3]{WngZll2}, which correspond to cases 4 and 5 in Table \ref{tableIA}, are $G$-unstable and compute the spectrum of $\lic_\pg$ together with the multiplicities.   

Let $l_1\leq l_2\leq l$ be non-negative integers and assume $2\leq l$.  For each $1\leq i\leq l$, we choose an irreducible symmetric space $G_i/K_i$ according to the following constraints: 
\begin{enumerate}[{\small $\bullet$}]
\item $\SO(2n_i)/\SO(n_i)\times\SO(n_i)$ or $\Sp(2n_i)/\Sp(n_i)\times\Sp(n_i)$ for any $1\leq i\leq l_1$ (Grassmannian spaces),

\item $\SO(n_i+1)/\SO(n_i)= S^{n_i}$ for any $l_1+1\leq i\leq l_2$ (spheres), 

\item $K_i$ is simple and $G_i/K_i$ is not a sphere (as above) for any $l_2+1\leq i\leq l$.
\end{enumerate} 
We will see below that the construction only depends on the spaces $G_i/K_i$ up to covering.
In Table~\ref{giki}, all possible $G_i/K_i$ are listed without repetition.  

\begin{remark}
The spaces $\SO(4)/\SO(2)\times\SO(2)$, $\Sp(2)/\Sp(1)\times\Sp(1)$, $\SU(4)/\Sp(2)$ and $(H\times H)/\Delta(H)$ with $H=\SU(2)$ or $\SO(3)$ were all omitted  in Table~\ref{giki} for being locally isometric to $\SO(3)/\SO(2)\times\SO(3)/\SO(2)$, ${\SO(5)}/{\SO(4)}$, $\SO(6)/\SO(5)$ and ${\SO(4)}/{\SO(3)}$, respectively.
\end{remark}

For each $G_i/K_i$, we consider the reductive decomposition $\ggo_i=\kg_i\oplus\mg_i$, set  $\nn_i=\dim{\mg_i}=\dim{G_i/K_i}$, and let $\pi_i:K_i\to\so(\mg_i)$ denote its isotropy representation.  
Consider the homogeneous space 
$$
\SO(\nn_i)/K_i:=\SO(\mg_i)/\pi_i(K_i).
$$ 
If $\sog(\nn_i)=\kg_i\oplus V_i$ is its reductive decomposition, then $\wedge^2 \mg_i \simeq \ad_{\fk_i}\oplus V_i$ as $K_i$-modules.  It turns out that $V_i=0$ if and only if $l_1+1\leq i\leq l_2$ (i.e.\ $G_i/K_i$ is a sphere).  Furthermore, as a representation of $K_i$, $V_i$ is irreducible for any $i>l_2$ and $V_i\simeq V_i^1\oplus V_i^2$ with $V_i^1,V_i^2$ inequivalent and irreducible submodules of the same dimension for any $i\leq l_1$ (see \S\ref{A.3a} and \S\ref{A.3b}).  Note that the each space $\SO(\nn_i)/K_i$ with $i\leq l_1$ corresponds to case 3 in Table~\ref{tableIA}.

We set  $G=G_1\times\dots\times G_l$ and $K=K_1\times\dots\times K_l$ and now consider the isotropy representation $\pi$ of $G/K$ given by  
$$
\mg:=\mg_1\oplus\dots\oplus\mg_l, \qquad \pi=\pi_1+\dots+\pi_l, \qquad \nn:=\dim{\mg}=\nn_1 +\dots+\nn_l.
$$ 
We study in this section the homogeneous space 
$$
M=\SO(\nn)/K=\SO(\mg)/\pi(K),  
$$ 
i.e., the case Table \ref{tableIA}, 5.

\begin{table}
{\small 
$$
\begin{array}{c|c|c|c|c}
G_i/K_i & \text{Cond.} &\dim\fk_i & \nn_i & \tfrac{\dim\fk_i}{\nn_i} 
\\[2mm] \hline \hline \rule{0pt}{14pt}
	\SO(2n)/\SO(n)\times\SO(n) & n\geq 3 & n(n-1) & n^2 & \tfrac{n-1}{n}
\\[2mm] \hline \rule{0pt}{14pt}
	\Sp(2n)/\Sp(n)\times\Sp(n) & n\geq 2 & 2n(2n+1) & 4n^2 & \tfrac{2n+1}{2n}
\\[2mm] \hline \rule{0pt}{14pt} 
	\SO(n+1)/\SO(n) & n\geq 2 & \tfrac{n(n-1)}{2} & n & \tfrac{n-1}{2}
\\[2mm] \hline \rule{0pt}{14pt}
	\SU(n)/\SO(n) & \underset{n\neq4}{n\geq3} & \tfrac{n(n-1)}{2} & \tfrac{(n-1)(n+2)}{2} & \tfrac{n}{n+2}
\\[2mm] \hline \rule{0pt}{14pt}
	\SU(2n)/\Sp(n) & n\geq 3 & n(2n+1) & (n-1)(2n+1) & \tfrac{n}{n-1}
\\[2mm] \hline \rule{0pt}{14pt}
	E_6/\Sp(4) & - & 36& 42 & \tfrac{6}{7}
\\[2mm] \hline \rule{0pt}{14pt}
	E_6/F_4 & -& 52& 26& 2
\\[2mm] \hline \rule{0pt}{14pt}
	E_7/\SU(8)& - & 63& 70& \tfrac{9}{10}
\\[2mm] \hline \rule{0pt}{14pt}
	E_8/\Spin(16) & - & 120&128&\tfrac{15}{16}
\\[2mm] \hline \rule{0pt}{14pt}
	F_4/\Spin(9) & - & 36 & 16 & \tfrac{9}{4}
\\[2mm] \hline \rule{0pt}{14pt}
	(H\times H)/\Delta H & \dim H>3 & \dim{H} & \dim{H} & 1
\\[2mm] \hline \hline
\end{array}
$$}
\caption{
	Irreducible symmetric spaces admitted as $G_i/K_i$ in \S\ref{som-sec}. 
	$H$ is any compact simple Lie group in the last row.
} \label{giki}
\end{table}

\begin{remark}
This space belongs to the case Table~\ref{tableIA}, 4 if and only if $G_i/K_i=H_i\times H_i/\Delta H_i$ for any $i=1,\dots,l$, where $H_1,\dots,H_l$ are compact simple Lie groups (in particular, $l_1=l_2=0$) and so $K\simeq H_1\times\dots\times H_l$, $m=\dim{\kg}$ and $(\mg,\pi)$ is equivalent to the adjoint representation of $K$.  On the other hand, setting $G_i/K_i=\SO(n+1)/\SO(n)$ for all $i=1,\dots,l$ gives $M=\SO(ln)/\SO(n)^l$, which are (for $l\geq3$) the cases Table~\ref{tableIA}, 1b and 2c for $n=2$ and $n\geq3$, respectively. 
\end{remark}

Consider the reductive decomposition $\sog(\nn)=\kg\oplus\pg$, where $\pg$ is the $K$-representation such that $\wedge^2\mg=\ad_\fk\oplus \fp$.
Thus $\pg$ decomposes in pairwise inequivalent $K$-irreducible subspaces as
\begin{equation}\label{p5}
\pg =
\bigoplus_{i\leq l_1} \big(\fp_{(ii)^1} \oplus \fp_{(ii)^2}\big)  \oplus 
\bigoplus_{i>l_2} \fp_{(ii)} \oplus
\bigoplus_{i<j} \fp_{(ij)},
\end{equation}
where 
$\pg_{(ii)^\epsilon}:=V_i^\epsilon \subset\sog(\nn_i)$ for $\epsilon=1,2$ and $i\leq l_1$,
$\pg_{(ii)}:=V_i\subset\sog(\nn_i)$ for $i>l_2$, and 
$\pg_{(ij)}:=\mg_i\otimes\mg_j$ for all $i<j$ (we set $\pg_{(ji)}:=\pg_{(ij)}$). 
In particular, $\SO(\nn)/K$ is multiplicity-free.
We have that $d_{(ii)^\epsilon}=\dim{V_i^\epsilon}$ for $i\leq l_1$ (recall that $d_{(ii)^1}=d_{(ii)^2}$), $d_{(ii)}=\dim{V_i}$ for $i>l_2$, $d_{(ij)}=\nn_i\nn_j$ for $i<j$, and the number of irreducible terms in the isotropy representation is given by 
\begin{equation}\label{eq-conceptual:r}
r=2l_1 +l-l_2+\tfrac{l(l-1)}{2}.
\end{equation}
The situation for $l_1=1, l_2=2, l=3$ can be graphically described as follows:
\begin{equation*}
\so(\nn)=\so(\nn_1+\nn_2+\nn_3)=
\left[\begin{matrix}
\fk_1\oplus \fp_{(11)^1}\oplus \fp_{(11)^2} & \fp_{(12)} & \fp_{(13)} \\
\fp_{(12)} & \fk_2 &  \fp_{(23)} \\
\fp_{(13)} & \fp_{(23)} & \fk_3\oplus \fp_{(33)}
\end{matrix}\right].
\end{equation*}
For any $i\leq l_2$ we also set $\fp_{(ii)}:=V_i$ and $d_{(ii)}:=\dim V_i$, so $d_{(ii)^\epsilon}=\tfrac{d_{(ii)}}{2}$ for any $i\leq l_1$ and
\begin{equation*}
	\dim{\kg_i}+d_{(ii)} = \tfrac{m_i(m_i-1)}{2},
	\qquad\forall i=1,\dots,l.
\end{equation*}
One can check that $[\pg_{(ii)},\pg_{(jj)}]_\pg\subset\delta_{i,j}\, \pg_{(ii)}$, $[\pg_{(ii)},\pg_{(ij)}]_\pg\subset\pg_{(ij)}$ if $i\ne j$, 
$[\pg_{(ij)},\pg_{(kl)}]_\pg=0$ if $k\ne l$ and $\{ i,j\}$ and $\{ k,l\}$ are either equal or disjoint, and $[\pg_{(ij)},\pg_{(jk)}]_\pg\subset\pg_{(ik)}$ if $\#\{i,j,k\}=3$. 
The only nonzero structural constants are therefore given by
\begin{equation}\label{ijk-som}
\begin{aligned}
&
[(ij)(jk)(ik)],\quad\forall\#\{i,j,k\}=3,\quad &&
[(ii)^1(ij)(ij)],\, [(ii)^2(ij)(ij)],\quad\forall i\leq l_1,\; i\neq j,
\\ &
[(ii)(ij)(ij)],\quad\forall i>l_2,\; i\neq j, &&
[(ii)^1(ii)^1(ii)^1],\, [(ii)^2(ii)^2(ii)^2],\quad\forall i\leq l_1, 
\\ &
[(ii)(ii)(ii)],\quad\forall i>l_2, 
&&
[(ii)^1(ii)^1(ii)^{2}],\, [(ii)^1(ii)^2(ii)^{2}], \quad\forall i\leq l_1.
\end{aligned}
\end{equation}

It is proved in \cite[Example 3]{WngZll2} that the Killing metric $g_{\kil}$ on $M=\SO(\nn)/K$ is Einstein if and only if 
\begin{equation}
\frac{\dim\fk_1}{\nn_1}=\dots= \frac{\dim\fk_l}{\nn_l},
\end{equation} 
which will be assumed from now on.  
According to Table \ref{giki}, the numbers $\kappa$ with $\kappa=\tfrac{\dim\fk_i}{\nn_i}$ having more than one solution are listed in Table~\ref{tabla-coincidencias}
(see also the table in \cite[pp.575]{WngZll2}).
Note that $\SO(4)/\SO(3)$ is present in the row $\tfrac{\dim\fk_i}{\nn_i}=1$ as $(\SU(2)\times\SU(2))/\Delta\SU(2)$.

\begin{table}
{
$$
\begin{array}{c@{\qquad}c}
\tfrac{\dim\fk_i}{\nn_i} &  G_i/K_i  \\ \hline \hline 
\rule{0pt}{14pt}
\tfrac67 &
	\tfrac{\SO(14)}{\SO(7)\times \SO(7)},  
	\tfrac{\SU(12)}{\SO(12)},
	\tfrac{E_6}{\Sp(4)}
\\[2mm] \hline \rule{0pt}{14pt}
\tfrac{9}{10}&
	\tfrac{\SO(20)}{\SO(10)\times \SO(10)},
	\tfrac{\SU(18)}{\SO(18)},
	\tfrac{E_7}{\SU(8)}
\\[2mm] \hline \rule{0pt}{14pt}
\tfrac{15}{16}&
	\tfrac{\SO(32)}{\SO(16)\times \SO(16)},
	\tfrac{\SU(30)}{\SO(30)},
	\tfrac{E_8}{\SO(16)}
\\[2mm] \hline \rule{0pt}{14pt}
\tfrac{n-1}{n}&
	\tfrac{\SO(2n)}{\SO(n)\times \SO(n)},
	\tfrac{\SU(2n-2)}{\SO(2n-2)}
	\quad\text{for } \underset{n\neq 7,10,16}{n\geq3}
\\[2mm] \hline \rule{0pt}{14pt}
1 &
	\tfrac{H_i\times H_i}{H_i} \text{ ($H_i$ simple)} 
\\[2mm] \hline \rule{0pt}{14pt}
\tfrac{2n+1}{2n}&
	\tfrac{\Sp(2n)}{\Sp(n)\times \Sp(n)},
	\tfrac{\SU(4n+2)}{\Sp(2n+1)}
	\quad\text{for }n\geq2
\\[2mm] \hline \rule{0pt}{14pt}
\tfrac32 &
	\tfrac{\SO(5)}{\SO(4)}, 
	\tfrac{\SU(6)}{\Sp(3)}
\\[2mm] \hline \rule{0pt}{14pt}
2 &
	\tfrac{\SO(6)}{\SO(5)}, 
	\tfrac{E_6}{F_4}
\\[2mm] \hline \hline 
\end{array}
$$}
\caption{Multiple solutions of $\kappa=\tfrac{\dim\fk_i}{\nn_i}$.}\label{tabla-coincidencias}
\end{table}

\begin{remark}\label{rem-caso5:[111]y[112]}
It follows immediately from Table~\ref{tabla-coincidencias} that $l_1(l_2-l_1)=0$, that is, Grassmannian spaces and spheres cannot take part simultaneously. 
Furthermore, $l_1>0$ (resp.\ $l_1<l_2$) forces $G_i/K_i=G_1/K_1$ for all $i\leq l_1$ (resp.\ $G_i/K_i=G_{l_2}/K_{l_2}$ for all $l_1<i\leq l_2$).  If $l_1>0$, we denote by $[111],[112],[122],[222]$ the structural constants of $\SO(\nn_1)/K_1$, which were considered in \S\ref{A.3a} and \S\ref{A.3b}.  Although we do not know their values in the case \S\ref{A.3b}, we know that $[111]=[222]$, $[112]=[122]$, and we have an expression for $[111]+3[112]$. 
\end{remark}

\begin{remark}\label{rem-conceptual:esferas}
Table~\ref{tabla-coincidencias} also implies that if $G_i/K_i=\SO(n+1)/\SO(n)$ with $n\neq 4,5$ for some $i$, then $G_i/K_i$ is independent of $i$ giving $M=\SO(ln)/\SO(n)^l$ already considered. 
\end{remark}

It is also shown in \cite[pp.574]{WngZll2} that the Casimir constant on any $\pg_{(ij)}$ relative to $-\kil_{\so(\nn)}|_{\fk}$ equals $a:=\tfrac{2\dim \fk_i}{\nn_i(\nn-2)}$ and, since the Einstein constant for $g_{\kil}$ is $\rho= \tfrac14+\tfrac12 a$ (see \eqref{ricgB}), one obtains that
\begin{equation}\label{eq-conceptual:rho}
	\rho= \frac14 + \frac{\dim\fk_i}{\nn_i(\nn-2)}, \qquad\forall i.
\end{equation}

\begin{lemma}\label{som1}
If $g_{\kil}$ is Einstein, then the nonzero structural constants of the homogeneous space $\SO(\nn)/K$ (see \eqref{ijk-som}) are given as follows: 
\begin{align*}
[(ii)(ii)(ii)]&= \tfrac{d_{(ii)}}{\nn-2} \left(\nn_i-2-\tfrac{4\dim\fk_i}{\nn_i}\right), \quad  i>l_2, 
\\ 
[(ii)^1(ii)^1(ii)^1]&= [(ii)^2(ii)^2(ii)^2]=\tfrac{\nn_i-2}{\nn-2} [111], \quad  i\leq l_1, 
\\ 
[(ii)^1(ii)^1(ii)^2]&= [(ii)^1(ii)^2(ii)^2]= \tfrac{\nn_i-2}{\nn-2} [112], \quad  i\leq l_1, 
\\ 
[(ij)(ij)(ii)] &= \tfrac{\nn_jd_{(ii)}}{\nn-2}, \quad i\ne j, \quad  i>l_2,  
\\  
[(ij)(ij)(ii)^1] &= [(ij)(ij)(ii)^2]= \tfrac{\nn_jd_{(ii)}}{2(\nn-2)}, \quad i\ne j, \quad  i\leq l_1,  
\\  
[(ij)(ik)(jk)]&= \tfrac{\nn_i \nn_j \nn_k}{2(\nn-2)}, \quad\#\{i,j,k\}=3.
\end{align*} 
\end{lemma}

\begin{proof}
Recall that the constants $[111]$ and $[112]$ were introduced in Remark~\ref{rem-caso5:[111]y[112]}.  We consider the usual $-\kil_{\so(\nn)}$-orthonormal basis of $\sog(\nn)$ (recall that $\kil_{\so(\nn)}(X,Y)=(\nn-2)\tr(XY)$) given by
$ 
\{X_{\alpha,\beta}:=\tfrac{1}{\sqrt{2(\nn-2)}} (E_{\alpha,\beta}-E_{\beta,\alpha}): 1\leq \alpha<\beta\leq \nn
\}, 
$ 
which provides, for each $\fp_{(ij)}$ with $i\neq j$, the $-\kil_{\so(\nn)}$-orthonormal basis  
$ 
\{
X_{\alpha,\beta}: 
1\leq \alpha - \widetilde \nn_i\leq \nn_{i}, \; 
1\leq \beta - \widetilde \nn_j\leq \nn_{j}
\},
$ 
where $\widetilde \nn_k := \nn_1+\dots+\nn_{k-1}$.  Using that $[X_{\alpha,\beta}, X_{\beta,\gamma}] 
= \tfrac{1}{\sqrt{2(\nn-2)}} X_{\alpha,\gamma}$ for all $\#\{\alpha,\beta,\gamma\}=3$, a straightforward computation gives the formula for $[(ij)(ik)(jk)]$.  

Suppose that $i\leq l_1$. 
Let $\{Y_h^{(ii)^\epsilon}\}_h$ be a $-\kil_{\so(\nn)}$-orthonormal basis of $\fp_{(ii)^\epsilon}=V_{i}^\epsilon\subset \so(\nn_i)\subset \so(\nn)$.
Since $\kil_{\so(\nn_i)} = \tfrac{\nn_i-2}{\nn-2} \kil_{\so(\nn)}|_{\so(\nn_i)}$, we have that $\left\{ \tfrac{\sqrt{\nn-2}}{\sqrt{\nn_i-2}} Y_h^{(ii)^\epsilon}\right\}_h$ is $-\kil_{\so(\nn_i)}$-orthonormal.  
This implies that  	
\begin{align*}
[(ii)^1(ii)^1(ii)^1] &= \tfrac{\nn_i-2}{\nn-2} [111], &
[(ii)^1(ii)^1(ii)^2] &= \tfrac{\nn_i-2}{\nn-2} [112], 
\\
[(ii)^2(ii)^2(ii)^2] &= \tfrac{\nn_i-2}{\nn-2} [222], &
[(ii)^1(ii)^2(ii)^2] &= \tfrac{\nn_i-2}{\nn-2} [122].
\end{align*}
The corresponding assertions for these terms follow since $[111]=[222]$ and $[112]=[122]$ (see Remark~\ref{rem-caso5:[111]y[112]}).  Similarly, for $i>l_2$, one has that 
$$
[(ii)(ii)(ii)] = \tfrac{\nn_i-2}{\nn-2} [111]_i, 
$$
where $[111]_i$ is the structural constant of the isotropically irreducible space $\SO(\nn_i)/K_i$, which can be computed using formulas \eqref{rhoci} and \eqref{rhoijk} for its Einstein constant as follows:  
$$
\frac14+\frac{\dim\fk_i}{\nn_i(\nn_i-2)}
= \frac12 - \frac{1}{4d_{(ii)}} [111]_i.  
$$
This proves the formula for $[(ii)(ii)(ii)]$.  

Let $p=(ii)^\epsilon$ for some $i\leq l_1$ or $p=(ii)$ for some $i>l_2$. 
We have that
\begin{align*}
[(ij)(ij)p]&
= 
\sum_{\substack{1\leq \alpha-\widetilde\nn_i\leq \nn_i \\ 1\leq \beta-\widetilde\nn_j\leq \nn_j }}
\sum_{\substack{1\leq \gamma-\widetilde\nn_i\leq \nn_i \\ 1\leq \delta-\widetilde\nn_j\leq \nn_j }}
\sum_{h=1}^{d_{p}} (-\kil_{\so(\nn)})\big([X_{\alpha,\beta}, X_{\gamma,\delta}],Y_h^p\big)^2
\\ 
&
= \frac{1}{2(\nn-2)}
\sum_{\substack{1\leq \alpha-\widetilde\nn_i\leq \nn_i \\ 1\leq \gamma-\widetilde\nn_i\leq \nn_i }}
\sum_{1\leq \beta-\widetilde\nn_j\leq \nn_j}
\sum_{h=1}^{d_{p}} (-\kil_{\so(\nn)})\big(X_{\alpha,\gamma} , Y_h^p\big)^2
\\ &
= \frac{\nn_j}{2(\nn-2)}
\sum_{\substack{1\leq \alpha-\widetilde\nn_i\leq \nn_i \\ 1\leq \gamma-\widetilde\nn_i\leq \nn_i }}
\sum_{h=1}^{d_{p}} (-\kil_{\so(\nn)})\big( X_{\alpha,\gamma} , Y_h^p\big)^2
= \frac{\nn_jd_{p}}{(\nn-2)},
\end{align*}
concluding the proof.  
\end{proof}

We are now ready to compute the $r\times r$ matrix $[\lic_\pg]$ of the Lichnerowicz Laplacian restricted to $\sca^2(M)^{\SO(\nn)}$ (see \eqref{Lpijk}).  Recall that the set of $r$ indexes, which will be denoted by $\II$, where $r=2l_1 +l-l_2+\tfrac{l(l-1)}{2}$, consists of  
$(ii)^1,(ii)^2$ for $1\leq i\leq l_1$, $(ii)$ for $l_2+1\leq i\leq l$ and $(ij)$ for each pair $1\leq i\leq j\leq l$.  

\begin{proposition}\label{som2}
The nonzero coefficients of the matrix $[\lic_\fp]$ are given by
\begin{align*}
[\lic_\fp]_{(ij),(ij)} &
=\tfrac{1}{(\nn-2)} 
\left(
\nn-\nn_i-\nn_j +
\tfrac{d_{(ii)}}{\nn_i}+ \tfrac{d_{(jj)}}{\nn_j} 
\right), \quad i\ne j, 
\\
[\lic_\fp]_{(ii),(ii)} &
= \tfrac{\nn-\nn_i}{\nn-2}, \quad  i>l_2, 
\\ 	
[\lic_\fp]_{(ii)^1,(ii)^1} &=[\lic_\fp]_{(ii)^2,(ii)^2}
=  \tfrac{\nn-\nn_i}{\nn-2}+\tfrac{4[112]}{d_{(ii)}} \tfrac{\nn_i-2}{\nn-2}, \quad  i\leq l_1, 
\\ 	
[\lic_\fp]_{(ii)^1,(ii)^2} &
=  -\tfrac{4[112]}{d_{(ii)}} \tfrac{\nn_i-2}{\nn-2}, \quad  i\leq l_1, 
\\ 	
[\lic_\fp]_{(ij),(ii)} &
=  -\tfrac{1}{\nn-2} \sqrt{d_{ii} \tfrac{\nn_j}{\nn_i} }
, \quad i\ne j, \quad  i>l_2,
\\
[\lic_\fp]_{(ij),(ii)^1} &=[\lic_\fp]_{(ij),(ii)^2}
=  -\tfrac{1}{\nn-2} \sqrt{\tfrac{d_{ii}}{2} \tfrac{\nn_j}{\nn_i} }
, \quad i\ne j, \quad  i\leq l_1,
\\
[\lic_\fp]_{(ij),(ik)} &
	= \tfrac{-\sqrt{\nn_j\nn_k}}{2(\nn-2)},
	\quad\#\{i,j,k\}=3.
\end{align*}
\end{proposition}

\begin{proof}
It follows from \eqref{Lpijk} and \eqref{ijk-som} that the only nonzero coefficients are
\begin{align*}
[\lic_\fp]_{(ij),(ij)} &
= \tfrac{1}{\nn_i\nn_j} 
\left(
[(ij)(ij)(ii)]+[(ij)(ij)(jj)]\right) +\tfrac{2}{\nn_i\nn_j}\sum_{k\neq i,j} [(ij)(ik)(jk)],
\quad\forall i\neq j,
\\
[\lic_\fp]_{(ii),(ii)} &
= \tfrac{1}{d_{(ii)}} \sum_{j\neq i}[(ii)(ij)(ij)], \quad\forall i>l_2,
\\
[\lic_\fp]_{(ii)^\epsilon,(ii)^\epsilon} &
= \tfrac{2}{d_{(ii)}} 
	\Big(
	[(ii)^{\epsilon}(ii)^{\epsilon+1}(ii)^{\epsilon}]
	+[(ii)^{\epsilon}(ii)^{\epsilon+1}(ii)^{\epsilon+1}]
	+\sum_{j\neq i}[(ii)(ij)(ij)]
	\Big), \quad\forall i\leq l_1,
\\
[\lic_\fp]_{(ii)^1,(ii)^2} &
= -\tfrac{2}{d_{(ii)}} 
\big(
	[(ii)^{1}(ii)^{2}(ii)^{1}]
	+[(ii)^{1}(ii)^{2}(ii)^{2}]
\big), \quad\forall i\leq l_1,
\\
[\lic_\fp]_{(ij),(ii)} &
=- \tfrac{[(ij)(ij)(ii)]}{\sqrt{d_{(ii)}}\sqrt{\nn_i\nn_j}},
\quad\forall i\neq j, \quad i>l_2,
\\ 
[\lic_\fp]_{(ij),(ii)^\epsilon} &
= -\tfrac{\sqrt{2}[(ij)(ij)(ii)^\epsilon]}{\sqrt{d_{(ii)}}\sqrt{\nn_i\nn_j}},
\quad\forall i\neq j, \quad i\leq l_1,
\\ 
[\lic_\fp]_{(ij),(ik)} &
= -\tfrac{1}{\nn_i\sqrt{\nn_j\nn_k}}  [(ij)(ik)(jk)],
\quad\forall\#\{i,j,k\}=3.
\end{align*}
Thus each of the formulas in the proposition can be proved by a straightforward computation using the expressions for the structural constants given in Lemma \ref{som1}.  
\end{proof}

The following proposition provides a large amount of eigenvectors of $\lic_\pg$.  
We recall from \S\ref{preli} that $[\lic_\fp]$ stands for the matrix of $\lic_\fp:\sym(\pg)^K\to \sym(\pg)^K$ with respect to the orthonormal basis $\{I_{r}\}_{r\in\II}$ of $\sym(\pg)^K$ given by $I_r|_{\fp_{q}} =\delta_{r,q}\tfrac{1}{\sqrt{d_r}}I_{\fp_r}$, where $I_{\fp_q}$ is the identity map on $\fp_q$. 

\begin{proposition}\label{som3}
Given indexes $1\leq i,j\leq l$ with $i\neq j$, let $A^{ij}=\sum_{r\in\II}a_r^{ij}I_r$ and $B^{ij}=\sum_{r\in\II}b_r^{ij}I_r$ be elements in $\sym(\pg)^K$ given by 
\begin{align*}
a_r^{ij}&=
\begin{cases}
	\nn_j \sqrt{d_{(ii)}} 
	& r=(ii),\,i>l_2,
	\\[2mm]
	\nn_j {\sqrt{d_{(ii)}/2}}
	& r=(ii)^\epsilon,\,i\leq l_1,
	\\[2mm]
	-\nn_i \sqrt{d_{(jj)}} 
	& r=(jj),\,j>l_2,
	\\[2mm]
	-\nn_i {\sqrt{d_{(jj)}/2}}
	& r=(jj)^\epsilon,\,j\leq l_1,
	\\[2mm]
	\frac{\nn_j-\nn_i}{2} \sqrt{\nn_i\nn_j} 
	& r=(ij),
	\\[2mm]
	\frac{\nn_j}{2} \sqrt{\nn_i\nn_h} 
	& r=(ih),\, h\neq i,j,
	\\[2mm]
	-\frac{\nn_i}{2} \sqrt{\nn_j\nn_h} 
	& r=(jh),\, h\neq i,j,
	\\[2mm]
	0& \text{otherwise,}
\end{cases}
&
b_r^{ij}&=
\begin{cases}
	\sqrt{d_{(jj)}} 
	& r=(ii),\,i>l_2,
\\[2mm]
	\sqrt{d_{(jj)}/2}
	& r=(ii)^\epsilon,\,i\leq l_1,
\\[2mm]
	\sqrt{d_{(jj)}} 
	& r=(jj),\,j>l_2,
\\[2mm]
	\sqrt{d_{(jj)}/2}
	& r=(jj)^\epsilon,\,j\leq l_1,
\\[2mm]
	\frac{-2\sqrt{d_{(ii)}d_{(jj)}}} {\sqrt{\nn_i\nn_j}} 
	& r=(ij),
\\[2mm]
	0& \text{otherwise.}
\end{cases}
\end{align*}
Then, $\lic_\fp(A^{ij})=\tfrac{m}{2(m-2)} A^{ij}$ and $\lic_\fp(B^{ij})= \lambda B^{ij}$, where
$$
\lambda
=\tfrac{\nn-1-\tfrac{2\dim\fk_i}{\nn_i}}{\nn-2}
= \tfrac{\nn-\nn_i+\tfrac{2d_{(ii)}}{\nn_i}}{\nn-2}
. 
$$
\end{proposition}

\begin{proof}
By applying Proposition~\ref{som2}, lengthy but straightforward computations give the assertions. 
We next give some details in a simple case. 

Suppose that $i,j> l_2$, so $B^{ij}= \sqrt{d_{(jj)}} I_{(ii)} + \sqrt{d_{(ii)}} I_{(jj)}
- {2\sqrt{d_{(ii)}d_{(jj)}}}/ {\sqrt{\nn_i\nn_j}} I_{(ij)}$.
Proposition~\ref{som2} and simple manipulations give
\begin{equation*}
\begin{aligned}
\lic_\fp(B^{ij})&
= \sum_{r\in\II}\left( 
	\sqrt{d_{(jj)}}  [\lic_\fp]_{r,(ii)} I_{r} 
	+ \sqrt{d_{(ii)}} [\lic_\fp]_{r,(jj)} I_{r} 
	- \tfrac{2\sqrt{d_{(ii)}d_{(jj)}}} {\sqrt{\nn_i\nn_j}} [\lic_\fp]_{r,(ij)} I_{r} 
\right)
\\ &
= 
I_{(ii)}
\tfrac{\sqrt{d_{(jj)}}}{\nn-2}
\left(
\nn-\nn_i
+ 2 \tfrac{d_{(ii)}}{\nn_i} 
\right)
+ I_{(jj)}
\tfrac{\sqrt{d_{(ii)}}}{\nn-2}
\left(
\nn-\nn_j
+ 2\tfrac{d_{(jj)}}{\nn_j}
\right)
\\ & \quad 
-I_{(ij)} \tfrac{2\sqrt{d_{(ii)}d_{(jj)}}}{\sqrt{\nn_i\nn_j}(\nn-2)} 
\left(
\nn
-\tfrac{\nn_j}2  
-\tfrac{\nn_i}2  
+ \tfrac{d_{(ii)}}{\nn_i} +\tfrac{d_{(jj)}}{\nn_j}
\right)
%
%
%
%
%
=\lambda B^{ij}.
\end{aligned}
\end{equation*}
It is a simple matter to extend the above proof when $i\leq l_1$ or $j\leq l_1$. 
Furthermore, $B^{ij}=0$ if $l_1<i\leq l_2$ or $l_1<j\leq l_2$, so the assertion holds trivially.
\end{proof}

We are finally in a position to prove the main result of this section.  

\begin{theorem}\label{som-main}
Assume that the standard metric $g_{\kil}$ on the homogeneous space $\SO(\nn)/K$ constructed above is Einstein.  
\begin{enumerate}[{\rm (a)}]
\item $g_{\kil}$ is always $G$-unstable and has coindex 
	$\geq l_1+(l-l_2)-1$.

\item If $l_1+(l-l_2)\geq2$ (i.e., at least two symmetric spaces different from a sphere are involved), then $g_{\kil}$ is a saddle point of $\scalar|_{\mca_1^G}$.

\item If $l_1=l_2=0$, then 
\begin{equation*}
\Spec(\lic_\fp) = 
\big\{
0,
\underbrace{\lambda_\fp ,\dots, \lambda_\fp}_{(l-1)\text{-times}} ,
\underbrace{\lambda_\fp^{\max} ,\dots, \lambda_\fp^{\max}}_{\binom{l}{2}\text{-times}} 
\big\},
\end{equation*}
where  $\lambda_\fp=\tfrac{\nn}{2(\nn-2)}$, $\lambda_\fp^{\max}= \tfrac{1}{(\nn-2)} (\nn-1-\tfrac{2\dim\fk_i}{\nn_i})$.
Moreover, $\lambda_\fp<2\rho<\lambda_\fp^{\max}$.  

\end{enumerate}
\end{theorem}

\begin{proof}
It follows from \eqref{eq-conceptual:rho} that $\tfrac{\nn}{2(\nn-2)}<2\rho$ if and only if 
$\tfrac12<\tfrac{\dim\fk_i}{\nn_i}$.  A simple inspection to Table~\ref{giki} yields that this condition always holds with the only exception of $G_i/K_i=\SO(3)/\SO(2)$ for all $i$, but this case can be omitted according to Remark~\ref{rem-conceptual:esferas}.   
	Proposition~\ref{som3} now implies that $g_{\kil}$ is $G$-unstable. 
	Moreover, the coindex is at least $l_1+(l-l_2)-1$ because $\{A^{ij}: j\leq l_1\text{ or }j>l_2, j\neq i\}$ is clearly linearly independent for any $i$. 

\begin{remark}	
Actually, the coindex is at least $l_1+(l-l_2)$ when $l_1<l_2$ by choosing $i$ such that $l_1<i\leq l_2$. 
\end{remark}

Under the assumption of part (b), Proposition~\ref{som3} ensures the existence of a second nonzero eigenvalue $\lambda$, which satisfies that $\lambda>2\rho$ if and only if $\tfrac{\dim\fk_i}{\nn_i} <\tfrac{\nn}{8}$, which always holds (see Table~\ref{giki}) proving part (b).

Concerning part (c), we first observe that the set $\{B^{ij}: i<j\}$ is clearly linearly independent, thus  the corresponding eigenvalue $\lambda$ has multiplicity at least $\binom{l}{2}$ by Proposition~\ref{som3}. 
Similarly, $\tfrac{\nn}{2(\nn-2)}$ has multiplicity at least $l-1$ because $\{A^{1i}: i>2\}$ is linearly independent. 
Since the kernel is at least $1$-dimensional and the size of $\lic_\fp$ is $\binom{l+1}{2}$, the spectrum has to have the form as stated.  
\end{proof}

\section{Case $E_7/\SU(2)^7$}\label{e7-sec}

This space is listed in Table \ref{tableIB2}, 7.  We consider the root space decomposition 
\begin{equation*}
\ggo_\C=\ft_\C \oplus\bigoplus_{ \alpha\in\Delta(\ggo_\C,\ft_\C)} \ggo_\alpha,
\end{equation*}
where
\begin{align*}
\Delta(\ggo_\C,\ft_\C)
=& \{\pm\ee_i\pm\ee_j: 1\leq i<j\leq 6\}
\cup \{\pm(\ee_7-\ee_8)\} \\
&\cup \left\{ \tfrac12\left(\textstyle{\sum\limits_{i=1}^6} (-1)^{n(i)}\ee_i\pm (\ee_7-\ee_8)\right): \textstyle{\sum\limits_{i=1}^6} n(i)\;\text{is odd}\right\}. 
\end{align*}
Inside the maximal subalgebra $\hg:=\so(12)\oplus\su(2)$ of maximal rank of $\eg_7$ associated to 
\begin{equation*}
\Delta(\hg_\C,\ft_\C) = \{\pm\ee_i\pm\ee_j: 1\leq i<j\leq 6\}\cup \{\pm(\ee_7-\ee_8)\},
\end{equation*}
we describe (the complexification of) the subalgebra $\kg=7\cdot\sug(2)$ as follows:
\begin{align*}
\kg_\C=&
\C H_{\ee_1-\ee_2}\oplus \ggo_{\ee_1-\ee_2} \oplus \ggo_{-(\ee_1-\ee_2)}
\oplus 
\C H_{\ee_1+\ee_2}\oplus \ggo_{\ee_1+\ee_2} \oplus \ggo_{-(\ee_1+\ee_2)}
\oplus
\\ &
\C H_{\ee_3-\ee_4}\oplus \ggo_{\ee_3-\ee_4} \oplus \ggo_{-(\ee_3-\ee_4)}
\oplus 
\C H_{\ee_3+\ee_4}\oplus \ggo_{\ee_3+\ee_4} \oplus \ggo_{-(\ee_3+\ee_4)}
\oplus
\\ &
\C H_{\ee_5-\ee_6}\oplus \ggo_{\ee_5-\ee_6} \oplus \ggo_{-(\ee_5-\ee_6)}
\oplus 
\C H_{\ee_5+\ee_6}\oplus \ggo_{\ee_5+\ee_6} \oplus \ggo_{-(\ee_5+\ee_6)}
\oplus
\\ &
\C H_{\ee_7-\ee_8}\oplus \ggo_{\ee_7-\ee_8} \oplus \ggo_{-(\ee_7-\ee_8)}, 
\end{align*}
where $H_{\alpha}$ is any non-trivial element in $[\mathfrak g_\alpha,\mathfrak g_{-\alpha}]\subset\mathfrak \tg_{\CC}$. Note that each triple $\C H_\alpha\oplus\ggo_\alpha\oplus\ggo_{-\alpha}$ is isomorphic to $\mathfrak{sl}_2(\mathbb C)$.  Thus the $\kil_{\eg_7}$-orthogonal complement $\pg$ of $\kg$ in $\ggo$ is given by $\fp=\fp_1\oplus\dots\oplus\fp_7$, where $(\fp_k)_\C= \bigoplus_{\alpha\in\Delta_k}\ggo_\alpha$ and
{\small  
\begin{align*}
\Delta_{1}&= \{\pm\ee_i\pm\ee_j: i=1,2,\; j=3,4\}, & 
\Delta_{4}&= \{\tfrac12 (\pm(\ee_1-\ee_2) \pm(\ee_3-\ee_4) \pm(\ee_5-\ee_6) \pm(\ee_7-\ee_8))\},\\ 
\Delta_{2}&= \{\pm\ee_i\pm\ee_j: i=1,2,\; j=5,6\}, & 
\Delta_{5}&= \{\tfrac12 (\pm(\ee_1+\ee_2) \pm(\ee_3+\ee_4) \pm(\ee_5-\ee_6) \pm(\ee_7-\ee_8))\}, \\ 
\Delta_{3}&= \{\pm\ee_i\pm\ee_j: i=3,4,\; j=5,6\}, & 
\Delta_{6}&= \{\tfrac12 (\pm(\ee_1-\ee_2) \pm(\ee_3+\ee_4) \pm(\ee_5+\ee_6) \pm(\ee_7-\ee_8))\}, \\ 
&&\Delta_{7}&= \{\tfrac12 (\pm(\ee_1+\ee_2) \pm(\ee_3-\ee_4) \pm(\ee_5+\ee_6) \pm(\ee_7-\ee_8))\}.
\end{align*}}

It is easy to check that each $\pg_i$ is $\Ad(K)$-irreducible and equivalent to the sum of four copies of  standard representations of four of the seven $\sug(2)$'s, which naturally produces the notation
$$
\pg = \pg_{(1234)} \oplus \pg_{(1357)} \oplus \pg_{(1256)} \oplus \pg_{(2457)} \oplus \pg_{(3456)} \oplus \pg_{(1467)} \oplus \pg_{(2367)}, 
$$
where every summand has dimension $16$ (see \cite[pp.579]{WngZll2}).  Note that each pair of these seven $4$-tuples of indexes intersects in exactly two numbers.  It turns out that the Lie bracket between any two summands only projects on the summand obtained after deleting their intersection (e.g., $[\pg_{(1234)},\pg_{(1357)}]\subset \pg_{(2457)}$) and the structural constants are all equal, say to a number $c$.  

It follows from \eqref{Lpijk} that $[\lic_\pg]_{pp}=\frac{3c}{8}$ and $[\lic_\pg]_{pq}=-\frac{c}{16}$ for all $p\ne q$, so
$$
\lic_\pg = \tfrac{7c}{16} I - \tfrac{c}{16}\left[
\begin{matrix}1&\dots&1\\ \vdots&\ddots&\vdots\\ 1&\dots&1\end{matrix}\right],
$$
and thus $\lic_\pg$ has a unique nonzero eigenvalue $\lambda_\pg=\tfrac{7c}{16}$ with multiplicity $6$.  

On the other hand, formula \eqref{rhoijk} implies that $2\rho =1-\tfrac{3c}{16}$.  It is not hard to show that $\kil_{\su(2)} = \tfrac19 \kil_{\fe_7}|_{\su(2)}$, so $\rho=\tfrac{1}{3}$ by \eqref{rhoci}, which gives   
$c=\tfrac{16}{9}$.  We conclude that  
$$
\lambda_\pg = \tfrac{7}{9} > \tfrac{2}{3} = 2\rho, 
$$
that is, $g_{\kil}$ is $G$-stable. 

\begin{remark}
The Casimir operator acts on any irreducible component $\chi_i=4\cdot\pi_1+3\cdot\pi_0$ of $\chi$ as $\cas_{\chi_i}=\frac49\cas_{\pi_1,-\kil_{\sug(2)}} =\frac49 \frac38 I_{\pg_i}=\frac16 I_{\pg_i}$ (here $\pi_1$ and $\pi_0$ respectively denote the standard and trivial representations of each of the seven copies of $\sug(2)$), which confirms that $\rho=\unc+\frac{1}{12}=\tfrac{1}{3}$ (see \eqref{ricgB}).   
\end{remark}

\section{Computing Einstein constants and the spectrum of $\lic_\pg$}\label{rho-sec} 

There are ten spaces for which it was not necessary to know $\rho$ nor $\lambda_\pg$ to obtain their $G$-stability types.  Except for the case Table \ref{tableIB1}, 4, they were all solved using the stability criteria given in \S\ref{Gstab-sec}.   In this section, for the standard metrics on these spaces, we compute the Einstein constants and in most cases also the full spectrum of $\lic_\pg$, as this may be useful in studying other problems.

\subsection{$\sog(3n+2)/\sog(n)\oplus\ug(n+1)$}\label{A.7b} 
(See Table \ref{tableIAA}, 7b. and case (ii) at the end of \S\ref{Gstab-sec}).  
The intermediate subalgebra 
$\fk\subset\hg:=\so(n)\oplus\so(2n+2)\subset\ggo$ gives that the only nonzero structural constant is $[122]$; indeed, $G/H$ and $H/K$ are both symmetric spaces.  One can check that $\rho=\tfrac{5}{12}$ by \eqref{rhoci}.  Now, \eqref{r2-rho} yields
$[122]=\tfrac{n(n+1)}{3}$ and so $\lambda_\pg=\unm$ by \eqref{r2-lambda}.

\subsection{$\eg_6/\sog(3)\oplus\sog(3)\oplus\sog(3)$}\label{B.2} 
(See Table \ref{tableIB1}, 2).  
	Let $\fh_\CC$ be the maximal subalgebra of $(\eg_6)_\CC$ constructed after deleting the only node with three edges in the extended Dynkin diagram. 
	It turns out that $\fk\subset \fh\simeq \sug(3)\oplus\sug(3)\oplus\sug(3)$, where $\sog(3)\subset \sug(3)$ is the standard embedding for each term, thus $\kil_{\sog(3)} = \tfrac16 \kil_{\sug(3)}|_{\sog(3)} = \frac{1}{24}\kil_{\fe_6}|_{\sog(3)}$.
Now formula \eqref{rhoci} implies that $\rho=\tfrac{5}{16}$.

\subsection{$\eg_6/\sug(2)\oplus\sog(6)$}\label{B.4'} 
(See Table \ref{tableIB1}, 4). 
	The intermediate subalgebra $\fk\subset\fh:= \su(6)\oplus\su(2)\subset\mathfrak g$ gives that the only nonzero structural constant is $[122]$.  Since $\kil_{\su(6)} = \tfrac12 \kil_{\fe_6}|_{\su(6)}$ (see \cite[pp.38]{WngZll2}) and  $\kil_{\so(2n)} = \tfrac{n-1}{2n} \kil_{\su(2n)}|_{\so(2n)}$, we obtain that  
	$\kil_{\so(6)} = \tfrac{1}{3} \kil_{\su(6)}|_{\so(6)} = \tfrac{1}{3}.\tfrac12 \kil_{\fe_6}|_{\so(6)}
		=\tfrac{1}{6}\kil_{\fe_6}|_{\so(6)}$,
and it is not hard to see that $\kil_{\su(2)} = \tfrac{1}{6} \kil_{\fe_6} |_{\su(2)}$.  Thus formula \eqref{rhoci} gives that $\rho=\tfrac38$.  

On the other hand, we have that $d_1=20$, $d_2=40$, thus $[122]=10$ by \eqref{r2-rho}.  
Now \eqref{r2-lambda} gives that $\lambda_\fp=\tfrac{3}{4}$.

\subsection{$\eg_8/\sog(9)$}\label{B.9} 
(See Table \ref{tableIB2}, 9).   
Formula \eqref{rhoci} gives that $\rho=\tfrac{13}{40}$ by using that  
	$
	\kil_{\sog(9)} = \tfrac{7}{18} \kil_{\sug(9)}|_{\sog(9)} = \tfrac{7}{18}\tfrac{3}{10}\kil_{\fe_8}|_{\sog(9)} =\tfrac{7}{60}\kil_{\fe_8}|_{\sog(9)}.
	$

\subsection{$\eg_8/4\cdot\sug(3)$}\label{B.12} 
(See Table \ref{tableIB2}, 12). 
	Let $\mathfrak h=\eg_6\oplus \su(3)$ be the maximal subalgebra of $\eg_8$.
	One has that $\mathfrak k\subset\mathfrak h$ by embedding $3\cdot\su(3)$ into $\eg_6$ as explained in \S\ref{B.2}.
We have that $r=4$ and $d_i=54$ for all $i=1,\dots,4$.  One can check that $\kil_{\sug(3)} = \tfrac{1}{10} \kil_{\fe_8}|_{\su(3)}$ for each of the four copies and so $\rho=\tfrac{19}{60}$ by \eqref{rhoci}.  Via a detailed construction of the $\pg_i$'s using the root system of $\eg_8$, it can be shown that 
	$[iii] = \tfrac{36}{5}$ for all $i$,
	$[ijk] = \tfrac{27}{5}$ if $\#\{i,j,k\}=3$, and $[ijk]=0$ otherwise.  
This in particular confirms the value $\rho=\tfrac{19}{60}$ using formula \eqref{rhoijk}.  

It follows from \eqref{Lpijk} that 
	$[\lic_\fp]_{k,k}= \tfrac{3}{5}$ and $[\lic_\fp]_{j,k}=-\tfrac{1}{5}$ for all $j\neq k$, 	
which gives 
\begin{equation*}
[\lic_\fp] = \tfrac{4}{5}\Id - \tfrac{1}{5}
	\begin{pmatrix}1&\dots&1\\ \vdots&\ddots&\vdots\\ 1&\dots&1\end{pmatrix}.  
\end{equation*}
Thus $\lambda_\fp=\tfrac{4}{5}$ and has multiplicity $3$.

\subsection{$\eg_8/4\cdot\sog(3)$}\label{B.13} 
(See Table \ref{tableIB2}, 13).   
	Let $\fh$ be the intermediate subalgebra $4\cdot \su(3)$ considered in \S\ref{B.12}. 
We have that
$
\kil_{\sog(3)} = \tfrac{1}{6} \kil_{\sug(3)}|_{\sog(3)} = \tfrac{1}{6}\tfrac{1}{10}\kil_{\fe_8}|_{\sog(3)} =\tfrac{1}{60}\kil_{\fe_8}|_{\sog(3)},
$  
which implies that 
$\rho=\frac{11}{40}$ by \eqref{rhoci}.

\subsection{$\eg_8/8\cdot\sug(2)$}\label{B.15} 
(See Table \ref{tableIB3}, 15).   
This case is very similar to the space treated in \S\ref{e7-sec}.  If we view $\kg$ as $4\cdot\sog(4)$, then $c_i=\tfrac{1}{15}$ for all $i=1,\dots,4$ since $\kil_{\sog(4)} = \tfrac{1}{5} \kil_{\sog(12)}|_{\sog(4)} = \tfrac{1}{15} \kil_{\fe_8}|_{\sog(4)}$.  According to \eqref{rhoci}, we obtain that $\rho=\tfrac{3}{10}$.

The subalgebra $\kg$ can be described in much the same way as in \S\ref{e7-sec} as a subalgebra of the maximal subalgebra $\sog(16)$ of $\eg_8$ by just adding the copy $\CC H_{e_7+e_8}\oplus\ggo_{e_7+e_8}\oplus\ggo_{-e_7-e_8}$.  There are $14$ $\Ad(K)$-irreducible subspaces of the same dimension $16$,   
\begin{align*}
\fp_{(1234)}, &&\fp_{(5678)}, && \fp_{(1256)}, && \fp_{(3478)}, && \fp_{(1278)},&& \fp_{(3456)},
&& \fp_{(1458)}, \\ 
\fp_{(2367)}, && \fp_{(1467)},&& \fp_{(2358)}, && \fp_{(1357)},&& \fp_{(2468)},&& \fp_{1368)},&& \fp_{(2457)},
\end{align*}
satisfying that 
\begin{equation*}
[\pg_{(a_1a_2a_3a_4)},\pg_{(b_1b_2b_3b_4)}] \subset
\begin{cases}
0 &\text{if }a_i\neq b_j \quad\forall i,j, \\
\pg_{(a_ia_jb_kb_l)} &\text{if }
\begin{cases}
	\{i,j\}=\{h:a_h\neq b_m\;\forall m\},\\
	\{k,l\}=\{h:b_h\neq a_m\;\forall m\}.
\end{cases}
\end{cases}
\end{equation*}
Moreover, all nonzero structural constant are the same, say $b$.  It follows easily from \eqref{Lpijk} that    
	$[\lic_\fp]_{k,k}=\tfrac{3b}{4}$ for all $k$ and, for $j\neq k$, $[\lic_\fp]_{j,k}=0$ if $[\fp_j,\fp_k]=0$ and $[\lic_\fp]_{j,k}=-\tfrac{b}{16}$ otherwise. Hence
$$
\lic_\fp = \tfrac{3b}{4}\Id - \tfrac{b}{16}
\left[\begin{matrix}
0&0&1&1&\dots&1&1\\
0&0&1&1&\dots&1&1\\
1&1&0&0&\dots&1&1\\
1&1&0&0&\dots&1&1\\
\vdots&\vdots&\vdots&\vdots& \ddots&\vdots&\vdots\\
1&1&1&1&\dots&0&0\\
1&1&1&1&\dots&0&0\\
\end{matrix}\right].  
$$
It is a simple matter to conclude that  
\begin{equation*}
\Spec(\lic_\fp) = \Big\{0, \underbrace{\lambda_\fp,\dots, \lambda_\fp}_{\textup{7-times}} ,\underbrace{\lambda_\fp^{\max},\dots, \lambda_\fp^{\max}}_{\textup{6-times}}\Big\}
\qquad\text{where}\quad \lambda_\fp=\tfrac{3b}{4},
\quad\lambda_\fp^{\max}=\tfrac{7b}{8}. 
\end{equation*}
On the other hand, since $\rho=\tfrac{3}{10}$, it follows from \eqref{rhoijk} that $b=\tfrac{16}{15}$ and so $\lambda_\fp=\tfrac{4}{5}$ and $\lambda_\fp^{\max}=\tfrac{14}{5}$.

\subsection{$\eg_8/\sog(5)\oplus\sog(5)$}\label{B.16} 
(See Table \ref{tableIB3}, 16).   
The intermediate subalgebra $\sog(5)\oplus\sog(5)\subset\sog(16)\subset\eg_8$ satisfies that $\kil_{\sog(16)} = \tfrac{7}{15} \kil_{\eg_8}|_{\sog(16)}$, and following the lines of \cite[pp.38--40]{DtrZll}, one can prove that $\kil_{\sog(5)} = \tfrac{3}{28} \kil_{\sog(16)}|_{\sog(5)}$, concluding that $\kil_{\sog(5)} = \tfrac{1}{20} \kil_{\eg_8}|_{\sog(5)}$.  Now formula \eqref{rhoci} gives that $\rho=\tfrac{7}{24}$.

\subsection{$\eg_8/\sug(3)\oplus\sug(3)$}\label{B.17} 
(See Table \ref{tableIB3}, 17).   
We argue in much the same way as in the previous case with the intermediate subalgebra $\sug(3)\oplus\sug(3)\subset\sug(9)\subset\eg_8$.  One has that $\kil_{\sug(9)} = \tfrac{3}{10} \kil_{\eg_8}|_{\sug(9)}$ and it can be shown that $\kil_{\sug(3)} = \tfrac{1}{9} \kil_{\sug(9)}|_{\sug(3)}$ by using \cite[pp.38--40]{DtrZll}.  Thus $\kil_{\sug(3)} = \tfrac{1}{30} \kil_{\eg_8}|_{\sug(3)}$ and so $\rho=\tfrac{17}{60}$ by \eqref{rhoci}.

\begin{table}
{\small 
$$
\begin{array}{c|c|c|c|c|c|c|c}
\text{No.} & \ggo/\kg & \rho & \lambda_\pg & \lambda_\pg^{\midop} & \lambda_\pg^{\max} & \text{C1} & \text{C2}
\\[2mm] \hline \hline \rule{0pt}{14pt}
1a.1 & \tfrac{\sug(3)}{\sg(3\cdot\ug(1))} & \tfrac{5}{12} & \unm & - & -& \text{No} & \text{No}
\\[2mm]  \hline \rule{0pt}{14pt}
1a.2 & \tfrac{\sug(4)}{\sg(4\cdot\ug(1))} & \tfrac{3}{8} & \unm & - & \tfrac{3}{4}& \text{No} & \text{No}
\\[2mm]  \hline \rule{0pt}{14pt}
1a.3 & \tfrac{\sug(n)}{\sg(n\cdot\ug(1))} & \tfrac{n+2}{4n} & \unm & - & \tfrac{n-1}{n}& \text{No} & \text{No}
\\[2mm]  \hline \rule{0pt}{14pt}
1b.1 & \tfrac{\sog(6)}{3\cdot\sog(2)} &\tfrac{3}{8} &  \unm^* & - & \tfrac{3}{4}^* & \text{No} & \text{No}
\\[2mm]  \hline \rule{0pt}{14pt}
1b.2 & \tfrac{\sog(2n)}{n\cdot\sog(2)} &\tfrac{n}{4(n-1)} & \tfrac{n}{2(n-1)}^* & \tfrac{n-2}{n-1}^* &1^*& \text{No} & \text{No}
\\[2mm]  \hline \rule{0pt}{14pt}
2a & \tfrac{\sug(nk)}{\sg(n\cdot\ug(k))} & \tfrac{n+2}{4n} & \unm & - & \tfrac{n-1}{n} &\text{No}&\text{No}
\\[2mm]  \hline \rule{0pt}{14pt}
2b & \tfrac{\spg(nk)}{n\cdot\spg(k)} & \tfrac{(n+2)k+2}{4(nk+1)} & \tfrac{nk}{2(nk+1)}& - &\tfrac{(n-1)k}{nk+1} &\text{No}&\text{No}
\\[2mm]  \hline \rule{0pt}{14pt}
2c & \tfrac{\sog(nk)}{n\cdot\sog(k)} & \tfrac{(n+2)k-4}{4(nk-2)} & \tfrac{nk}{2(nk-2)}&- &\tfrac{(n-1)k}{nk-2} &\text{No}&\text{No}
\\[2mm]  \hline \rule{0pt}{14pt}
 3a  & \tfrac{\sog(n^2)}{\sog(n)\oplus\sog(n)} & \tfrac{n^3+2n-4}{4n(n^2-2)} & \tfrac{n^2-4}{n^2-2} & - & - & \text{No} & \text{No}
\\[2mm]  \hline \rule{0pt}{14pt}
 3b  & \tfrac{\sog(4n^2)}{\spg(n)\oplus\spg(n)} & \tfrac{2n^3+n+1}{4n(2n^2-1)} & & & & \text{No} & \text{No}
\\[2mm]  \hline \rule{0pt}{14pt}
4 & \tfrac{\sog(n)}{\kg} & \tfrac{n+2}{4(n-2)} & \tfrac{n}{2(n-2)}^* & - &1^*& \text{No}&\text{No}
\\[2mm]  \hline \rule{0pt}{14pt}
5 & \tfrac{\sog(\nn)}{\kg_1\oplus\dots\oplus\kg_l} &  \unc+\tfrac{\dim{\kg_i}}{\nn_i(\nn-2)} & \tfrac{\nn}{2(\nn-2)}^*  & - & \tfrac{\nn_i(\nn-1)-\dim{\kg_i}}{\nn_i(\nn-2)}^* 
&\text{No}&\text{No}
\\[2mm]  \hline \rule{0pt}{14pt}
6 & \tfrac{\sug(pq+l)}{\sug(p)\oplus\sug(q)\oplus\ug(l)} & \tfrac{p^2q^2+3p^2+3q^2+1}{4(p^2q^2+p^2+q^2+1)}^* 
& 
\frac{p^2q^2+p^2+q^2+3}{2(p^2q^2+p^2+q^2+1)}^* & - & - & \text{No} &\text{No}
\\[2mm]  \hline \rule{0pt}{14pt}
7a & \tfrac{\spg(3n-1)}{\spg(n)\oplus\ug(2n-1)} & \tfrac{5}{12}^* & \unm^* &- & -& \checkmark  '& \text{No}
\\[2mm]  \hline \rule{0pt}{14pt}
7b & \tfrac{\sog(3n+2)}{\sog(n)\oplus\ug(n+1)} & \tfrac{5}{12}^* & \unm^* &- & -& \checkmark &\text{No}
\\[2mm]  \hline \rule{0pt}{14pt}
8 & \tfrac{\sog(26)}{\spg(1)\oplus\spg(5)\oplus\sog(6)} & \tfrac{29}{80}^* & \tfrac{21}{40}^* &- &- &\text{No} & \text{No} 
\\[2mm]  \hline \rule{0pt}{14pt}
9 & \tfrac{\sog(8)}{\ggo_2} & \tfrac{5}{12} & & & & \checkmark^* & \checkmark^*
%
\\[2mm] \hline\hline
\end{array}
$$}
\caption{\cite[Table IA, pp.577]{WngZll2}.  Einstein constants and spectra of $\lic_\pg$.  In case $7a$, $\checkmark'$ means that the criterion C1 works for any $n\geq 3$, it gives $2\rho\leq\lambda_\pg$ for $n=2$ and it does not work if $n=1$.  $\checkmark^*$ means that the criteria only imply that $\lambda_\pg^{\max}\leq 2\rho$.  
In case 5, we know that the information on the spectra is valid assuming that $l_1=l_2=0$ 
(see Theorem \ref{som-main}).}\label{tableIAA} 
\end{table}

\subsection{Full flag manifolds}\label{ff} 
(See Table \ref{tableIAA}, 1a.3, 1b.1, 1b.2 and Table \ref{tableIB3}, 18a, 18b, 18c).  
Let $G$ be a compact simple Lie group and consider a maximal torus $T\subset G$ with Lie algebra $\tg\subset\ggo$.   We study in this subsection the full flag manifold $M=G/T$.  As usual, we have the root space decomposition 
\begin{equation*}
\ggo_\C = \ft_\C\oplus \bigoplus_{\alpha\in\Delta^+}
	\ggo_{\alpha}\oplus \ggo_{-\alpha}, 
\end{equation*}
and for each $\alpha\in\Delta$, we take $E_\alpha\in \ggo_{\alpha}$ such that $\kil(E_\alpha,E_{-\alpha})=1$ and $[E_{\alpha},E_{-\alpha}]=H_\alpha$, where $H_\alpha\in \ft_\C$ is defined by $\kil(H,H_\alpha)=\alpha(H)$ for all $H\in\ft_\C$.  The $-\kil_\ggo$-orthogonal reductive decomposition for $G/T$ is given by 
\begin{equation}
		\ggo=\ft\oplus \bigoplus_{\alpha\in\Delta^+} \fp_\alpha,
\end{equation}
where $\pg_\alpha$ is $\Ad(T)$-invariant and irreducible and it is generated by the $-\kil_\ggo$-orthonormal basis 
$$
	X_1^{\alpha} := \tfrac{1}{\sqrt 2} (E_\alpha-E_{-\alpha}),\qquad 
	X_2^{\alpha} := \tfrac{\mi}{\sqrt 2} (E_\alpha+E_{-\alpha}).
$$
Note that $G/T$ is therefore multiplicity-free, $r=\#\Delta^+=\tfrac{\dim{\ggo}-\dim{\tg}}{2}$ and $d_1=\dots=d_r=2$.  Using that $[E_\alpha, E_\beta]=N_{\alpha,\beta} E_{\alpha+\beta}$ and that these numbers satisfy that $N_{\alpha,\beta}=-N_{-\alpha,-\beta}$ and $N_{\alpha,\beta}=-N_{\beta,\alpha}$, it is straightforward to show (see \cite[p.~89]{Bhm}, \cite[Section 3]{Skn} or \cite[Proposition 2.3]{ArvChrSkn}) that for all $\alpha,\beta\in\Delta^+$
\begin{equation}\label{eq-fullflag:[p_alpha,p_beta]}
[\fp_\alpha,\fp_\beta] = 
	\fp_{\alpha+\beta}\oplus \fp_{\alpha-\beta},
\end{equation}
where $\fp_\gamma=0$ if $\gamma\notin\Delta$ and $\fp_{-\gamma}=\fp_{\gamma}$ if $\gamma\in \Delta^+$. 
Moreover, for all $\alpha,\beta, \gamma\in\Delta^+$, the structural constant is given by 
\begin{equation}
\label{eq-fullflag:[alphabeta(alpha+beta)]}
[\alpha\beta\gamma] = 
\begin{cases} 
2(N_{\alpha,\beta})^2 & \quad\text{if } \gamma=\alpha+\beta, \\ 
2(N_{\alpha,-\beta})^2 & \quad\text{if either} \; \gamma=\alpha-\beta \; \text{or} \; \gamma=-\alpha+\beta,\\ 
0 & \quad \text{otherwise}.  
\end{cases}
\end{equation}

\begin{table}
{\small 
$$
\begin{array}{c|c|c|c|c|c}
& \sug(n) & \sog(2n) & \eg_6 & \eg_7 & \eg_8 
\\[2mm] \hline \hline \rule{0pt}{14pt}
\kappa_\ggo &2(n-2) & 4(n-2) & 20 & 32 & 56 
\\[2mm] \hline \rule{0pt}{14pt}
\bb_\ggo & \tfrac{1}{n} & \tfrac{1}{2(n-1)} & \tfrac{1}{12} & \tfrac{1}{18} & \tfrac{1}{30} 
%
\\[2mm] \hline\hline
\end{array}
$$}
\caption{Lie theoretical invariants.} \label{table-ff}
\end{table}

It follows from \eqref{ricgB} that $g_{\kil}$ is Einstein on $G/T$ if and only if all the roots have the same length, which holds precisely in the following cases: 
$$
\SU(n)/T^{n-1}, \qquad \SO(2n)/T^n, \qquad E_6/T^6, \qquad E_7/T^7, \qquad E_8/T^8.
$$ 
We assume from now on that $G/T$ is one of these spaces.  Thus all the nonzero structural constants $[\alpha\beta(\alpha\pm\beta)]$ are equal to a number $\bb_\ggo$ (see \cite[(9)]{Skn} or \cite[Remark 2.4]{ArvChrSkn}).  On the other hand, given $\alpha\in\Delta^+$, it is easy to check that the number  
\begin{equation*}
	\kappa_\ggo:=\#\{\beta\in\Delta^+: \alpha+\beta\in\Delta \text{ or } \alpha-\beta\in\Delta\},
\end{equation*}
does not depend on $\alpha$ and its value is as in Table \ref{table-ff}.  It follows from \eqref{rhoijk} that $\rho= \unm-\tfrac{\kappa_\ggo\bb_{\ggo}}{8}$, and since alternatively, $\rho=\unc+\tfrac{\dim{\tg}}{2d}$ by \eqref{rhoci}, we obtain the value of $\bb_\ggo$ as given in Table \ref{table-ff}.  In particular, the Einstein constant $\rho$ of $g_{\kil}$ is respectively given by 
$
\tfrac{n+2}{2n},  \tfrac{n}{2(n-1)},  \tfrac{7}{12}, \tfrac{5}{9},  \tfrac{8}{15}.  
$
We now consider the $r\times r$ matrix $A_\ggo:=[a_{\alpha,\beta}]_{\alpha,\beta\in\Delta^+}$, where $a_{\alpha,\beta}=1$ if $\alpha+\beta\in\Delta$ or $\alpha-\beta\in\Delta$ and $a_{\alpha,\beta}=0$ otherwise.  

\begin{lemma}\label{sec-fullflag:L_p2rho}
	The Lichnerowicz Laplacian of $(G/T,g_{\kil})$ restricted to $\tca\tca_{g_B}^G$ is given by 
	\begin{equation*}
[\lic_\fp]= \tfrac{\bb_{\ggo}}{2} \left(\kappa_\ggo I - A_\ggo \right).  
	\end{equation*}
\end{lemma}
\begin{proof}
	It follows from \eqref{Lpijk} that 
	$[\lic_\fp]_{\alpha,\alpha}=\tfrac{\kappa_\ggo\bb_{\ggo}}{2}$ for all $\alpha\in\Delta^+$ and, for $\alpha\neq \beta$ in $\Delta^+$, 
	\begin{align*}
		[\lic_\fp]_{\alpha,\beta} &
		=
		\begin{cases}
			-\tfrac{\bb_{\ggo}}{2}&\text{if }\alpha+\beta\in\Delta \text{ or } \alpha-\beta\in\Delta, \\
			0&\text{otherwise,}
		\end{cases}
	\end{align*}
	concluding the proof.  
\end{proof}

In order to compute $\Spec(L_\fp)$, it only remains to know $\Spec(A_\ggo)$, which has already been computed for $\sug(n)$ and $\sog(2n)$ in \cite{stab-tres} and \S\ref{so2n-sec}, respectively.   The exceptional cases were worked out with the software Sage~\cite{Sage}. 

\begin{lemma}\label{lem-fullflag:Spec(A_g)}
The spectrum $\Spec(A_\ggo)$ of $A_\ggo$ is given as follows: 
$\{2, -1,-1\}$ if $\ggo=\sug(3)$,
$\{4, -2,-2,0,0,0\}$ if $\ggo=\sog(6)$,
\begin{align*}
	\{2(n-2), \underbrace{n-4,\dots,n-4}_{(n-1)\text{-times}}, \underbrace{-2,\dots,-2}_{\tfrac{n(n-3)}{2}\text{-times}} \}
	&\qquad\text{if $\ggo=\sug(n)$, $n\geq4$,}
	\\
	\{4(n-2), \underbrace{2(n-4),\dots,2(n-4)}_{(n-1)\text{-times}}, \underbrace{-4,\dots,-4}_{\tfrac{n(n-3)}{2}\text{-times}} , \underbrace{0,\dots,0}_{\tfrac{n(n-1)}{2}\text{-times}} \}
	&\qquad\text{if $\ggo=\sog(2n)$, $n\geq4$,}
	\\
	\{20, \underbrace{2,\dots,2}_{20\text{-times}}, \underbrace{-4,\dots,-4}_{15\text{-times}} \}
	&\qquad\text{if $\ggo=\eg_6$,}
	\\
	\{32, \underbrace{4,\dots,4}_{27\text{-times}}, \underbrace{-4,\dots,-4}_{35\text{-times}} \}
	&\qquad\text{if $\ggo=\eg_7$,}
	\\
	\{56, \underbrace{8,\dots,8}_{35\text{-times}}, \underbrace{-4,\dots,-4}_{84\text{-times}} \}
	&\qquad\text{if $\ggo=\eg_8$.}
\end{align*}
\end{lemma}

The following corollary follows from the above two lemmas.  

\begin{corollary}
The spectrum of the Lichnerowicz Laplacian of $(G/T,g_{\kil})$ restricted to $\tca\tca_{g_B}^G$ is given as in Table \ref{tableIAA}, 1 and Table \ref{tableIB3}, 18.
\end{corollary}

\begin{remark}
Alternatively, the number $\bb_\ggo$ (see Table \ref{table-ff}), which has been computed in \cite{Skn} in all the classical cases, can be obtained as follows for $\eg_6$, $\eg_7$ and $\eg_8$.  For $\ggo=\fe_8$, we have 
$$
		\Delta(\ggo_\C,\ft_\C)
		= \{\pm\ee_i\pm\ee_j: 1\leq i<j\leq 8\}
		\cup \left\{ \tfrac12 \sum_{i=1}^8 (-1)^{n(i)} \ee_i: \sum_{i=1}^8 n(i)\;\text{is even}\right\}, 
$$
and let $\so(12)$ denote the subalgebra of $\ggo$ attached to the subset  $\{\pm\ee_i\pm\ee_j: 1\leq i<j\leq 6\}$.  It follows from \cite[pp.37]{DtrZll} that $\kil_{\fe_7}= \tfrac35 \kil_{\fe_8}|_{\fe_7}$ and $\kil_{\so(12)} = \tfrac59 \kil_{\fe_7}|_{\so(12)}$, so $\kil_{\so(12)} = \tfrac59 \tfrac35 \kil_{\fe_8}|_{\so(12)}
=\tfrac13 \kil_{\fe_8}|_{\so(12)}$.  This implies that 
$$
			\bb_\ggo 
			= \tfrac{1}{27}\sum_{i,j,k} (-3\kil_{\so(12)})\big(
			[\sqrt{3}X_i^{\ee_1-\ee_2},  \sqrt{3}X_j^{\ee_2-\ee_3}],  \sqrt{3}X_k^{\ee_1-\ee_3} \big)^2
			=\tfrac13\, \bb_{\so(12)}= \tfrac{1}{30}.
$$
In much the same way, for $\ggo=\fe_7$ one considers the same subalgebra $\so(12)$ and use that 
$\kil_{\so(12)} = \tfrac59 \kil_{\fe_7}|_{\so(12)}$ and for $\ggo=\fe_6$, 
the subalgebra $\so(10)$ attached to $\{\pm\ee_i\pm\ee_j: 1\leq i<j\leq 5\}$, for which one has that $\kil_{\so(10)} = \tfrac23 \kil_{\fe_6}|_{\so(10)}$.   
\end{remark}

\end{document}